\documentclass[11pt, reqno]{amsart}

\usepackage{enumitem}
\usepackage{subfiles}
\usepackage[margin=1in]{geometry} 
\usepackage{amsmath,amsthm,amssymb, multicol, array, mathrsfs, esint, yhmath, verbatim, dsfont}

\usepackage[dvipsnames]{xcolor}
\usepackage[linktocpage=true]{hyperref}
\hypersetup{
	colorlinks=true,
	linkcolor=blue,
	citecolor=blue,
	urlcolor=blue
}
\usepackage{listings}

\usepackage{hyperref,cleveref}

\newcommand{\eps}{\varepsilon}

\newcommand{\ty}{\widetilde{y}}
\newcommand{\tu}{\widetilde{u}}

\newcommand{\subtitle}[1]{%
  \posttitle{%
    \par\end{center}
    \begin{center}\large#1\end{center}
    \vskip0.5em}%
}

\theoremstyle{plain}
\newtheorem{thm}{Theorem}[section]
\newtheorem{theorem}[thm]{Theorem}
\newtheorem*{thm*}{Theorem}
\newtheorem{lemma}[thm]{Lemma}
\newtheorem{proposition}[thm]{Proposition}

\newtheorem*{lem*}{Lemma}

\newtheorem*{ex*}{Example}
\newtheorem{definition}{Definition}[section]
\newtheorem{corollary}[thm]{Corollary}
\newtheorem{remark}{Remark}

\usepackage[framemethod=default]{mdframed}

\numberwithin{equation}{section}
 
 \newcommand{\be}{\begin{equation}}
\newcommand{\ee}{\end{equation}}

\newcommand{\mR}{\mathbb{R}}
\newcommand{\mC}{\mathbb{C}}
\newcommand{\mN}{\mathbb{N}}

\newcommand{\hy}{\hat y}

\newcommand{\hu}{\hat u}

\title{On the local null controllability of a viscous Burgers' system in finite time}

\author[H.-M. Nguyen]{Hoai-Minh Nguyen}
\address[H.-M. Nguyen]{Sorbonne Universit\'e, Universit\'e Paris Cit\'e, CNRS, INRIA, \newline \indent 
Laboratoire Jacques-Louis Lions, LJLL, F-75005 Paris, France
}
\email{hoai-minh.nguyen@sorbonne-universite.fr}

\author[M.-N. Tran]{Minh-Nguyen Tran}
\address[M.-N. Tran]{Sorbonne Universit\'e, Universit\'e Paris Cit\'e, CNRS, INRIA, \newline \indent 
Laboratoire Jacques-Louis Lions, LJLL, F-75005 Paris, France
}
\email{minh-nguyen.tran@sorbonne-universite.fr}

\begin{document}

\begin{abstract}
This paper is devoted to the local null controllability of the Burgers control system $y_t - y_{xx} + y y_x = u(t)$ on a bounded interval imposed by the zero Dirichlet boundary condition. It is known from the work of Marbach \cite{Marbach18} that this control system is not locally null controllable in small time. In this paper, we prove that the system is not locally null controllable in {\it finite time} as well. Our approach is inspired by the works of Coron, Koenig, and Nguyen \cite{Coron-Koenig-Nguyen}, and Nguyen \cite{Nguyen} on the controllability of the KdV system and is different from the one of Marbach. 
\end{abstract}


\maketitle

\tableofcontents

\section{Introduction}

\subsection{Introduction and statement of the main result}

In this paper, we study the local null controllability in finite time of a viscous Burgers control system using internal controls depending only on {\it time}. More precisely, for $T > 0$, we consider the following control system:
        \be
        \begin{cases}   \label{ViscousBurgers1D}
            y_t (t, x) - y_{xx} (t, x) + y y_x (t, x)= u(t) \quad &\text{ in } (0, T) \times (0, 1), \\[5pt]
            y(t, 0) = y(t, 1) = 0 &\text{ in } (0, T), \\[5pt]
            y(0, x) = y_0(x) &\text{ in } (0, 1),
        \end{cases}
        \ee
where $y_0$ is the initial datum and $u$ is the control. The control is thus internal and depends only on time. 

For $T>0$, set 
\begin{align}   \label{YT}
Y_T = C([0, T], L^2(0, 1)) \cap L^2((0, T), H^1(0, 1)),
\end{align}
and equip this space with the following norm 
\begin{align}	\label{norm-YT}
\| y \|_{Y_T} = \| y \|_{C([0, T], L^2(0, 1))} + \| y\|_{L^2((0, T), H^1(0, 1))}.
\end{align}
It is known that for every $y_0 \in L^2(0, 1)$ and $u \in L^1(0, T)$, there exists a unique weak solution $y \in Y_T$ of \eqref{ViscousBurgers1D} (see, e.g., \Cref{lem-ap1} in \Cref{wp-bg}).

\medskip
We first recall several notions on the local null controllability of \eqref{ViscousBurgers1D} considered in this paper. The first one is on the null controllability in time $T>0$. 

\begin{definition} Let $T>0$. System \eqref{ViscousBurgers1D} is said to be \textit{locally null controllable in time $T$} if for every $\varepsilon > 0$, there is $\delta > 0$ such that for every $y_0 \in L^2(0, 1)$ with $\| y_0\|_{L^2(0, 1)} < \delta$, there exists a real-valued function $u \in L^1(0, T)$ with $\| u \|_{L^1(0, T)} < \varepsilon$ such that
\begin{align*}
    y(T, \cdot) = 0,
\end{align*}
where $y \in Y_T$ is the unique solution of $(\ref{ViscousBurgers1D})$.
\end{definition}

Here are two related notions. 

\begin{definition} System \eqref{ViscousBurgers1D} is said to be small-time locally null controllable if for every $T>0$, the system is locally null controllable in time $T$. 
\end{definition}

\begin{definition} System \eqref{ViscousBurgers1D} is said to be finite-time locally null controllable if there exists $T>0$ such that the system is locally null controllable in time $T$.
\end{definition}

We next discuss the controllability of the linearized system of \eqref{ViscousBurgers1D} around 0, which  is given by 
\begin{align}   \label{linearizedBurgers}
\begin{cases}
    y_{t} (t, x) - y_{xx} (t, x) = u(t) \quad &\text{ in } (0, T) \times (0, 1), \\[6pt]
    y(t, 0) = y(t, 1) = 0 &\text{ in } (0, T), \\[6pt]
    y(0, x) = y_0(x) & \mbox{ in } (0, 1). 
\end{cases}
\end{align}
Let $\{e_k \}_{k \in \mN^*}$ be the Hilbert basis of $L^2(0, 1)$ formed by the eigenfunctions of the Laplacian operator on the interval $(0, 1)$, namely
\begin{align} \label{def-ek}
    e_k (x) = \sqrt{2} \sin (k \pi x) \quad \mbox{ for } k \in \mN^* \text{ and } x \in (0, 1).  
\end{align}
Denote by $\mathscr{M}$ the  subspace  of $L^2(0, 1)$ given by 
\begin{align}   \label{unreachable space}
    \mathscr{M} = \overline{\mathrm{span} \left\{e_k \; | \; k \in \mathbb{N}^*, \; k  \text{ even} \right\}}^{\| \cdot \|_{L^2(0, 1)}}.
\end{align}
For each $k \in \mathbb{N}^*$, set 
\begin{align} \label{def-phik}
\phi_k(t, x) = e^{\pi^2 k^2 t} \sin(k \pi x) \quad \mbox{ for } (t, x) \in [0, + \infty) \times [0, 1]. 
\end{align}
Then 
\begin{align} \label{prop-phik}
\begin{cases}
 \phi_{k,t} (t, x) +  \phi_{k,xx} (t, x) = 0  &\mbox{ in } (0, +\infty) \times (0, 1), \\[6pt]
\phi_k (t, 0) = \phi_k(t, 1) = 0 &\mbox{ in } [0, + \infty). 
\end{cases}
\end{align}

Given $y_0 \in L^2(0, 1)$ and $u \in L^1(0, T)$, let $y \in Y_T$ be the unique weak solution of  \eqref{linearizedBurgers}. By integration by parts, we derive from \eqref{linearizedBurgers} and \eqref{prop-phik} that 
\begin{align}   \label{ipp-linear}
    e^{ \pi^2 k^2 T} \int_0^1  y(T, x)  \sin(k \pi x)  dx - \int_0^1 y_0(x)  \sin(k \pi x)  dx =  \int_0^T \int_0^1 u(t) \phi_k(t, x)  dx  dt.
\end{align}
Since the right-hand side of \eqref{ipp-linear} is 0 when $k$ is even, it follows from \eqref{ipp-linear} that 
\be
\mbox{ if $y(T, \cdot) = 0$, then  $y_0 \in \mathscr{M}^{\bot}$}.
\ee 
Therefore, for every $y_0 \in \mathscr{M} \setminus \{0\}$ and for every $T>0$, there does not exist any control $u \in L^1(0, T)$ steering system \eqref{linearizedBurgers} from $y_0$ at time $0$ to $0$ at time $T$. Consequently, the linearized system \eqref{linearizedBurgers} is not null controllable in any positive time.

\medskip 
It was shown by Marbach \cite{Marbach18} that \eqref{ViscousBurgers1D} is not locally null controllable in small time using controls with small norms in $L^2(0, T)$. The analysis in \cite{Marbach18} is somehow indirect.  
A change of variables is first used to obtain a Burgers control system with small viscosity in a fixed time interval $(0, 1)$. The viscosity of the new system is equivalent to the time in the initial one. The result obtained is then derived from the fact that the new system is not locally controllable for small viscosity. The starting point of the analysis of the new system, whose linearized system is not controllable, is the so-called power series expansion method introduced by Coron and Cr\'epeau \cite{Coron-Crepeau}. The obstruction of the controllability is on the second order.  The analysis of this part is quite involved and uses the theory of weakly singular integral operators where the smallness of the viscosity plays an essential role. The choice of the test function for the obstruction in \cite{Marbach18} is somehow mysterious for us.

\medskip 
In this paper, we prove that the Burgers system \eqref{ViscousBurgers1D} is not even locally null controllable in {\it finite time} using controls with small norms in $[H^{\frac{3}{4}}(0, T)]^*$ \footnote{In this paper, for $s > 0$ and for an open interval $I$ of $\mR$,  $[H^s(I)]^*$ denotes the dual space of $H^s(I)$ with the corresponding norm.}. Our analysis is inspired from the work of Coron, Koenig, and Nguyen \cite{Coron-Koenig-Nguyen}, and Nguyen \cite{Nguyen} on the controllability of the KdV system. The analysis is thus different from and more direct than the one used in \cite{Marbach18}. The norm used for the smallness of the control in this paper is better than the one considered in \cite{Marbach18}.  

\medskip 
Here is the main result. 

 \begin{theorem}    \label{Mainresult}
       For every $T > 0$, there is $\varepsilon_T > 0$ such that for every $\delta > 0$, there exists $y_0 \in L^2(0, 1)$ with {$\|y_0\|_{L^2(0, 1)} < \delta$} such that for all $u \in L^1(0, T)$ with $\|u\|_{[H^{\frac{3}{4}}(0, T)]^*} < \varepsilon_T$, we have
            \begin{align*}
                y(T, \cdot) \neq 0,
            \end{align*}
            where $y \in {{Y}_T}$ is the unique solution of $\eqref{ViscousBurgers1D}$.
    \end{theorem}

A  version of \Cref{Mainresult} containing also the information of $y_0$ is later given in \Cref{coercivity->obstruction}.

\subsection{Ideas of the proof}
 The starting point of the analysis is the power series expansion method. We formally write the control and the solution under the form 
\begin{align*}
     u = \eps u_1 + \eps^2 u_2 + \ldots, \quad y  = \eps y_1 + \eps^2 y_2 + \ldots. 
 \end{align*}
 Then 
\begin{align*}
    y y_x = \eps^2 y_{1} y_{1, x} + \ldots.
\end{align*}
This motivates us to consider the following systems: 
\begin{align}   
    \label{our1storder-***}
    \begin{cases}   
        y_{1, t} (t, x) - y_{1, xx} (t, x) = u_1(t) \quad &\text{ in } (0, +\infty) \times (0, 1), \\[6pt]
        y_{1}(t, 0) = y_{1}(t, 1) = 0 &\text{ in } (0, +\infty),\\[6pt]
        y_1(0, \cdot) = 0 &\text{ in } (0, 1),
    \end{cases}
\end{align}
and
\begin{align}   
    \label{our2ndorder-***}
    \begin{cases}   
        y_{2, t} (t, x) -y_{2, xx} (t, x) + y_1 y_{1, x} (t, x) = u_2(t) \quad &\text{ in } (0, +\infty) \times (0, 1), \\[6pt]
        y_{2}(t, 0) = y_{2}(t, 1) = 0 &\text{ in } (0, +\infty), \\[6pt]
        y_2(0, \cdot) = 0 &\text{ in } (0, 1),
    \end{cases}
\end{align}  
where $u_1, u_2 \in L^1(0, +\infty)$. 

Our analysis is then inspired by the work of Coron, Koenig, and Nguyen \cite{Coron-Koenig-Nguyen}, and Nguyen \cite{Nguyen}. The key observation/analysis is that, if $\mathrm{supp} \; u_1 \subset [0, T]$ and
\begin{align}	\label{y10->0}
y_1(T, \cdot ) = 0 \mbox{ in } (0, 1),
\end{align} 
for some $T>0$, then for some even $k_0$ (for example, one can take $k_0 \in \{2, 10\}$), there exists a positive constant $C_T$ depending only on $T$ such that 
\begin{align}	\label{our-obstruct}
\int_0^1 y_2(T, x) \left(-\sin(k_0 \pi x) \right) dx  \geq  C_T \| u_1 \|_{[H^{\frac{5}{4}}(0, T)]^*}^2. 
\end{align}

\medskip

We next describe briefly how to prove \eqref{our-obstruct}. Let $k \in \mN^*$ be even. By integration by parts, one has  
\begin{align} \label{standard-identity}
\int_0^1 y_2(T, x) \phi_k(T, x)  dx =  \frac{1}{2} \int_0^T \int_0^1 y_1^2(t, x) \phi_{k, x}(t, x)  dx  dt.  
\end{align}
Note that $u_2$ is not involved in \eqref{standard-identity} since $k$ is even. 

In this paper, for an appropriate function $v$ defined on $(0, +\infty) \times (0, 1)$, we extend $v$ by $0$ on $(-\infty, 0) \times (0, 1)$ and still denote this extension by $v$. Set 
\begin{align}   
    \hat{v}(z, x) = \frac{1}{\sqrt{2 \pi}} \int_{\mathbb{R}} e^{-i z t} v(t, x) dt = \frac{1}{\sqrt{2 \pi}} \int_0^{+\infty} e^{-i z t} v(t, x) dt \quad \text{ for } (z, x) \in \mathbb{C} \times (0, 1), 
\end{align}     
which is the Fourier transform of $v$ with respect to the time variable $t$.  A similar definition is also used for a function defined on $(0, +\infty)$.

 The main analysis of \eqref{our-obstruct} is to estimate the right-hand side of \eqref{standard-identity}. As in \cite{Coron-Koenig-Nguyen} and \cite{Nguyen}, we first take the Fourier transform in time of \eqref{our1storder-***} and combine with \eqref{y10->0} to obtain a formula for $\hat y_1$ as a function of $\hu_1$ (see \Cref{lem-y1}). More importantly, as in \cite{Nguyen}, we then use Parseval's identity to obtain the following core formula (see \Cref{lem-key-1}):  
\begin{align} \label{key-1}
 \int_0^T \int_0^1 y_1^2(t, x) \phi_{k, x}(t, x) dx  dt = k \pi \int_{\mathbb{R}} \left| \hu_1\left(z + i \frac{(k \pi)^2}{2} \right) \right|^2 \Phi_k \left(z + i\frac{(k \pi)^2}{2} \right) dz,
\end{align}
where $\Phi_k$ is a \textit{real-valued} function independent of $u_1$ (see \eqref{defPhik} for the definition of $\Phi_k$). The fact that $\Phi_k$ is a real-valued function is crucial in our analysis and is one of the new key ingredients of \cite{Nguyen} in comparison with  \cite{Coron-Koenig-Nguyen}.  We then study the asymptotic behavior and the sign of the function $\Omega_k$ given by
\begin{align} \label{def-Omega}
	\Omega_k(z) : = \Phi_k \left(z + i \frac{(k \pi)^2}{2} \right) \mbox{ for } z \in \mathbb{R}.
\end{align}

The key observation/analysis of the paper is to show that, for $k_0 \in \{2,  10\}$, there exists a positive constant $C$ such that
\begin{align}\label{key-2}
-\Omega_{k_0}(z) \geq C (|z|^2 + 1)^{-\frac{5}{4}} \quad \forall z \in \mR 
\end{align} 
(see \Cref{as-behavior}). As a consequence of  \eqref{key-2}, we derive from  \eqref{key-1} that
\begin{align}	\label{key-1+2}
	\int_0^T \int_0^1 y_1^2(t, x) \phi_{k_0, x}(t, x) dx dt \leq - C_T \|u_1\|^2_{[H^{\frac{5}{4}}(0, T)]^*}
\end{align}
(see \Cref{2obstruction}). We then obtain \eqref{our-obstruct} as a direct consequence of \eqref{def-phik}, \eqref{standard-identity}, and \eqref{key-1+2}. 

The proof of \eqref{key-2} is inspired from \cite[Theorem 1.3]{Nguyen} where the controllability of the KdV equation was considered. In comparison with \cite{Nguyen}, we can establish the sign of $\Omega_2(z)$ and $\Omega_{10}(z)$ for all $z \in \mR$. Its proof is involved and requires detailed analysis of $\Phi_k$. 
 
To complete the analysis of \Cref{Mainresult} from \eqref{our-obstruct}, using the approach proposed in  \cite{Coron-Koenig-Nguyen} and \cite{Nguyen}, we need to establish several {\it new} estimates for solutions of the heat equation (see \Cref{pro-y1} in \Cref{subsec-2ineq}). In particular, concerning the heat system \eqref{linearizedBurgers} with $y_0 = 0$, we show that, for some positive constant $C$ depending only on $T$,   
\begin{align}		\label{intro-y}
	\|y\|_{L^2((0, T), H^1(0, 1))} \leq C \|u\|_{[H^{\frac{3}{4}}(0, T)]^*},
\end{align}    
where $u \in L^1(0, T)$ and $y \in Y_T$ is the corresponding solution of \eqref{linearizedBurgers}. The standard estimate uses $\| u \|_{L^1(0, T)}$ instead of $\|u\|_{[H^{\frac{3}{4}}(0, T)]^*}$ (see, e.g., \Cref{ap-heat-fx}). The proof of  \eqref{intro-y} is via a spectral decomposition and uses a technical lemma (\Cref{lem-spectral}), which is interesting in itself. The proof of this lemma involves a clever application of Fubini's theorem and Parseval's identity.
\subsection{Literature review}

The Burgers equation was introduced by Burgers \cite{Burgers} and shares many features with the Navier-Stokes system. Many works have been devoted to the control aspects of the viscous Burgers equation. Let us first briefly discuss the null controllability aspect of this equation. When a Dirichlet boundary control is applied on one side (with no internal controls), Fursikov and Imanuvilov \cite{Fursikov-Imanuvilov} showed that the system is locally null controllable in small time and globally null controllable in a finite time that depends on the size of the initial datum. Their proof relies on Carleman estimates for the parabolic problem, viewing the nonlinear term $y y_x$ as a small forcing term. Fernandez-Cara and Guerrero \cite{CG-07} later demonstrated that such a system is not globally null controllable in small time. The same phenomenon holds even when Dirichlet boundary controls are used on both sides, as shown by Guerrero and Imanuvilov \cite{Guerrero-Imanuvilov}. Coron \cite{Coron-07} established that with a Dirichlet boundary control on one side, the finite time required for global null controllability of the Burgers system does not depend on the initial data. His analysis is based on the maximum principle and the local result. The situation changes when an internal control, which depends only on time, is added to the Dirichlet boundary controls on both sides. Chapouly \cite{Chapouly}, using the return method introduced by Coron \cite{Coron-93, Coron-96} (see also \cite[Chapter 6]{Coron-07-book}), proved that in this case, the Burgers equation is globally null controllable in small time. Marbach \cite{Marbach14} improved this result by showing that the same conclusion holds when using only a Dirichlet boundary control on the left and an internal control depending only on time. Robin \cite{Robin} has obtained similar results to those in \cite{Marbach14} for the generalized Burgers equation. Other results on the controllability of the viscous Burgers equation using boundary or internal controls can be found in \cite{FI-95, Diaz-96, CG-05, GS-07}. 

The method used in this paper is inspired by the work of Coron, Koenig, and Nguyen \cite{Coron-Koenig-Nguyen}, and Nguyen \cite{Nguyen} on the controllability of the KdV system, with the starting point being the power series expansion method. This method was introduced by Coron and Cr\'epeau \cite{Coron-Crepeau} to study the controllability of the KdV system using Neumann boundary controls on the right for critical lengths whose unreachable space of the linearized system is of dimension 1. Later,  Cerpa \cite{Cerpa07}, and Cr\'epeau and Cerpa \cite{CC09} applied the power series expansion method to prove finite-time controllability of such a system at all critical lengths. Coron, Koenig, and Nguyen \cite{Coron-Koenig-Nguyen} showed that the local null controllability in small time does not hold for a class of critical lengths. The critical lengths in this case were discovered by Rosier \cite{Rosier97}. Nguyen \cite{Nguyen} implemented the power series expansion method to study the controllability of KdV equations using Dirichlet boundary controls on the right for all critical lengths. The set of critical lengths in this case was observed by Glass and Guerrero \cite{GG10}. The power series expansion method was also used to study the controllability of the water tank problem \cite{Coron-Koenig-Nguyen-WT}, as well as of the Schr\"odinger equation using bilinear controls \cite{BC06, Beauchard08, BM14}. It was also applied to study the decay of solutions of KdV systems at critical lengths for which the energy of the solutions of the linearized system is conserved \cite{Ng-decay}. 

\subsection{Organization of the paper} The paper is organized as follows. In Section \ref{sect-key-1}, we establish a formula for $\hat y_1$ in \Cref{lem-y1} and derive \eqref{key-1} in \Cref{lem-key-1}. In Section \ref{sect-key-2}, we study $\Omega_k$ in \Cref{as-behavior} which particularly gives \eqref{key-2} for $k_0 \in \{2, 10\}$. In Section \ref{sect-obstruction}, we prove \eqref{key-1+2} in  \Cref{2obstruction}. In Section \ref{subsec-2ineq}, we establish \eqref{intro-y} in \Cref{pro-y1}. In Section \ref{provemainsresult}, we prove \Cref{coercivity->obstruction}, which implies \Cref{Mainresult}. Several standard and technical results are given in the appendices.

\section{On the linearized system and proof of (\ref{key-1})} \label{sect-key-1}

We begin this section with a formula for $\hy_1$ where $y_1$ is a solution of \eqref{our1storder-***} that satisfies \eqref{y10->0}.

\begin{lemma}   \label{lem-y1}
Let $u_1 \in L^1(0, +\infty)$. Let $y_1 \in C([0, +\infty), L^2(0, 1)) \cap L^2_{loc}((0, +\infty), H^1(0, 1))$ be the unique weak solution of \eqref{our1storder-***}. Suppose that $\mathrm{supp} \; u_1 \subset [0, T]$ and $y_1(T, \cdot) = 0$ for some $T>0$. For $z \in \mC \setminus \{0\}$ and for $x \in (0, 1)$, it holds  
 \begin{align}	\label{Fourier-ODE}
    		\hy_1(z, x) = \frac{i \hu_1(z)}{z} \left(\frac{e^{\lambda_2(z)} - 1}{e^{\lambda_2(z)} - e^{\lambda_1(z)}} e^{\lambda_1(z) x} + \frac{1 - e^{\lambda_1(z)}}{e^{\lambda_2(z)} - e^{\lambda_1(z)}} e^{\lambda_2(z) x} - 1 \right).
    \end{align}
\end{lemma}


Here and in what follows, for $z \in \mathbb{C}$, we denote by $\lambda_1(z), \lambda_2(z)$ the two square roots of $iz$ with the convention that $\mathrm{arg}(\lambda_1(z)) \in (-\pi/2, \pi/2]$ for $z \in \mC \setminus \{0\}$. 

\begin{proof}  We derive from \eqref{our1storder-***} that for every $z \in \mC$,
\begin{align}\label{lem-y1-p1}
    \begin{cases}    	
    		\hy_1(z, \cdot) \in H^2(0, 1) \cap H^1_0(0, 1), \\[6pt]      
        (iz) \hy_1(z, \cdot) - (\hy_1)_{xx} (z, \cdot) = \hu_1 (z).
    \end{cases}
    \end{align}

\medskip

We search for a function $w: \mC \setminus \{0\} \times [0, 1] \to \mC$ such that for every $z \in \mC \setminus \{0\}$, 
\begin{align}	\label{lem-y1-p1'}
	\begin{cases} 
		w(z, \cdot) \in H^2(0, 1) \cap H^1_0(0, 1), \\[6pt]
		(iz) w(z, \cdot) - w_{xx}(z, \cdot) = \hu_1(z). 
	\end{cases}
\end{align}
Let $z \in \mC \setminus \{0\}$. Suppose that $w(z, \cdot)$ has the form 
\begin{align*}
w(z, x) = a (z) e^{\lambda_1(z) x} + b (z) e^{\lambda_2(z) x} + \frac{\hu_1(z)}{iz} \quad \text{ for } x \in [0, 1].
\end{align*}
Since $w(z, 0) = w(z, 1) = 0$, we have 
$$
a(z) + b(z) =  a(z)e^{\lambda_1(z)} + b(z)e^{\lambda_2(z)} = \frac{i\hu_1(z)}{z},  
$$
which yields 
$$
a(z) = \frac{i \hu_1(z)}{z} \frac{e^{\lambda_2(z)} - 1}{e^{\lambda_2(z)} - e^{\lambda_1(z)}} \quad \mbox{ and } \quad b(z) = \frac{i \hu_1(z)}{z} \frac{1- e^{\lambda_1(z)}}{e^{\lambda_2(z)} - e^{\lambda_1(z)}}. 
$$
We deduce that, for $x \in [0, 1]$, 
 \begin{align}	\label{lem-y1-p2}
    		w(z, x) = \frac{i \hu_1(z)}{z} \left(\frac{e^{\lambda_2(z)} - 1}{e^{\lambda_2(z)} - e^{\lambda_1(z)}} e^{\lambda_1(z) x} + \frac{1 - e^{\lambda_1(z)}}{e^{\lambda_2(z)} - e^{\lambda_1(z)}} e^{\lambda_2(z) x} - 1 \right).
    \end{align}
It is clear that for every $z \in \mC \setminus \{0\}$, $w(z,\cdot)$ given in \eqref{lem-y1-p2} satisfies \eqref{lem-y1-p1'}.

\medskip

It remains to show that for every $z \in \mC \setminus \{0\}$,
\begin{align}	\label{hy1=w}
	\hy_1(z, x) = w(z, x) \quad \forall x \in [0, 1],
\end{align}
where $w$ is given in \eqref{lem-y1-p2}. Indeed, it follows from \eqref{lem-y1-p1} and \eqref{lem-y1-p1'} that \eqref{hy1=w} holds for every $z$ outside a discrete subset of $\mC$ (more precisely, \eqref{hy1=w} holds for every $z \in \mC \setminus \{0\}$ such that $-iz$ belongs to the resolvent set of the operator $\mathcal{A}: H^2(0, 1) \cap H^1_0(0, 1) \to L^2(0, 1), \mathcal{A}f := - f''$).  In addition, using the argument in the proof of \cite[Lemma A.1]{Coron-Koenig-Nguyen}, we have that for each $x \in [0, 1]$, $w(\cdot, x)$ is an analytic function with respect to $z$ in $\mC \setminus \{0\}$. Note also that for $x \in [0, 1]$, $\hat y_1(\cdot, x)$ is an entire function with respect to $z$. It follows that \eqref{hy1=w} holds for every $z \in \mC \setminus \{0\}$. The proof is complete. 
\end{proof}

For $z \in \mC \setminus \{0\}$ and $k \ge 1$, set 
\begin{align} \label{def-J}
\Psi_k(z) = \int_0^1 \left| \left(e^{\lambda_2(z)} - 1 \right) e^{\lambda_1(z) x} + \left(1 - e^{\lambda_1(z)} \right) e^{\lambda_2(z) x} + e^{\lambda_1(z)} - e^{\lambda_2(z)} \right|^2 \cos(k \pi x) dx,
\end{align}
and
\begin{align}
\label{defPhik}
        \Phi_k(z)
        = \frac{1}{
        |z|^2 |e^{\lambda_1(z)} - e^{\lambda_2(z)}|^2} \Psi_k(z). 
\end{align}

We have the following result which yields \eqref{key-1}.  

\begin{lemma}   \label{lem-key-1}
Let $u_1 \in L^1((0, +\infty), \mathbb{R})$ and let $y_1 \in C([0, +\infty), L^2(0, 1)) \cap L^2_{loc}((0, +\infty), H^1(0, 1))$ be the unique solution of \eqref{our1storder-***}. Suppose that $\mathrm{supp} \; u_1 \subset [0, T]$ and $y_1(T, \cdot) = 0$ for some $T>0$. Then, for every $k \in \mN^*$,    
\begin{align}   \label{projy2}  
       \int_0^T \int_0^1 y_1^2(t, x) \phi_{k, x}(t, x) dx  dt = k \pi \int_{\mathbb{R}} \left| \hu_1\left(z + i \frac{(k \pi)^2}{2} \right) \right|^2 \Phi_k \left(z + i\frac{(k \pi)^2}{2} \right) dz,
  \end{align}
  where $\phi_k$ and $\Phi_k$ are given in \eqref{def-phik} and \eqref{defPhik}, respectively. 
\end{lemma}

\begin{proof} Extend $y_1$ by 0 for $(t, x) \in (-\infty, 0) \times (0, 1)$ and still denote this extension by $y_1$. Since $y_1(t, x) = 0$ for $t > T$ and for $t < 0$,  we have
\begin{align}	\label{lem-key-1-p1}
	\int_0^T \int_0^1 y^2_1(t, x) \phi_{k, x}(t, x) dx dt = k \pi \int_{\mR} \int_0^1 y_1^2(t, x) e^{\pi^2 k^2 t} \cos(k \pi x) dx dt.
\end{align}
By Parseval's identity and the fact that $y_1$ is real-valued (since $u_1$ is assumed to be real-valued as well), 
    \begin{align*}
    \int_{\mathbb{R}} \int_0^1 y_1^2(t, x) e^{\pi^2 k^2 t} \cos (k \pi x) dx dt 
	&=  \int_0^1  \cos (k\pi x) \int_{\mathbb{R}} \left|y_1(t, x) e^{\frac{\pi^2 k^2 t}{2}} \right|^2 dt   dx \\
    &=  \int_0^1  \cos(k \pi x) \int_{\mathbb{R}} \left|\widehat{y_1(\cdot, x) e^{\frac{\pi^2 k^2 \cdot}{2}}} (z) \right|^2 dz   dx, 
    \end{align*}
which implies
    \begin{align}	\label{lem-key-1-p2}
       \int_{\mathbb{R}} \int_0^1 y_1^2(t, x) e^{\pi^2 k^2 t} \cos (k \pi x) dx dt = \int_{\mathbb{R}} \int_0^1 \left|\hy_1 \left(z + i \frac{(k \pi)^2}{2}, x \right) \right|^2 \cos (k \pi x) dx  dz.
    \end{align}
The conclusion now follows from \eqref{lem-key-1-p1}, \eqref{lem-key-1-p2}, and \Cref{lem-y1}. 
\end{proof}

\section{Properties of the multiplier \texorpdfstring{$\Omega_k$}{Omega k}} \label{sect-key-2}

In this section, for $k \in \mN^*$, we study the function $\Omega_k: \mR \to \mR$ given by
$$
 \Omega_k (z)  = \Phi_k \left(z + i\frac{(k \pi)^2}{2} \right) \quad \text{ for } z \in \mR, 
$$
where $\Phi_k$ is defined in \eqref{defPhik}.  The goal then is to derive properties related to the sign of the left-hand side of \eqref{projy2}. The case $k_0 \in \{2, 10\}$ is of special interest. For notational convenience, set, for $k \in \mN^*$ and $z \in \mathbb{R}$,
\begin{align}\label{def-Lambda}	
	\Lambda_{1,k}(z) = \lambda_{1}\left(z + i \frac{(k \pi)^2}{2} \right), \quad \Lambda_{2,k}(z) =  \lambda_2\left(z + i \frac{(k \pi)^2}{2} \right), 
\end{align}
and 
\begin{align}\label{def-I}
\Theta_k(z) = \Psi_k\left(z + i \frac{(k \pi)^2}{2} \right),
\end{align}
where $\Psi_k$ is defined by \eqref{def-J}.  Recall that $\lambda_1(z), \lambda_2(z)$ are the two square roots of $iz$ with the convention that $\mathrm{arg}(\lambda_1(z)) \in (-\pi/2, \pi/2]$ for $z \in \mC \setminus \{0\}$. 

 It follows from \eqref{defPhik} that, for $z \in \mR$,
 \begin{align}	\label{Omegakformula}
        \Omega_k(z) = \frac{1}{\left(z^2 + \frac{(k \pi)^4}{4} \right) \left|e^{\Lambda_{1,k}(z)} - e^{\Lambda_{2,k}(z)} \right|^2} \Theta_k(z). 
 \end{align}
Note that, by \eqref{def-J}, for $z \in \mR$, 
 \begin{multline}	\label{Thetakformula}
\Theta_k(z) \\[6pt]
= \int_0^1 \left| \left(e^{\Lambda_{2,k}(z)} - 1 \right) e^{\Lambda_{1,k}(z) x} + \left(1 - e^{\Lambda_{1,k}(z)} \right) e^{\Lambda_{2,k}(z) x} + e^{\Lambda_{1,k}(z)} - e^{\Lambda_{2,k}(z)} \right|^2 \cos(k \pi x) dx. 
 \end{multline}

Here is the main result of this section, which implies \eqref{key-2}.

\begin{proposition} \label{as-behavior}  \label{sign}  \label{even+continuous} 
Let $k \in \mN^*$  be {\bf even}. The following properties hold. 
\begin{enumerate}
	\item[1)] \label{Oeven} The function $\Omega_k$ is continuous and even, i.e.,  $\Omega_k(z) = \Omega_k(-z)$ for  $z \in \mR$.
	 
    \item[2)] \label{Oasym} We have
    \begin{align}  
       \Omega_k(z) = -\sqrt{2} |z|^{-\frac{5}{2}} + O \left(|z|^{-\frac{7}{2}} \right) \text{ as }  |z| \to +\infty.
    \end{align}
    	\item[3)] We have
    	\begin{align}
    		\Omega_k(z) < 0 \quad \mbox{ for } z \in \mR \mbox{ with }  |z| > \frac{3(k \pi)^2}{2 \sqrt{7}}.
    	\end{align}
    \item[4)] \label{Onegative} If $k_0 \in \{2, 10\}$, then
    \begin{align}
        \Omega_{k_0}(z) < 0 \quad \forall z \in \mathbb{R}.
    \end{align}
\end{enumerate}
\end{proposition}

Here and in what follows, for two functions $f: I \to \mathbb{C}$ and $g: I \to [0, + \infty)$ ($I = [0,+\infty)$ or $I = \mR$), the notation $f(z) = O(g(z))$ means that there exist positive constants $C$ and $M$, which only depend on $k$, such that $|f(z)| \leq C g(z)$ for all  $z \in I$ satisfying $|z| \geq M$. 

\medskip 
The proof of \Cref{as-behavior} essentially uses the following property of $\Theta_k$.

\begin{proposition} \label{prop-Theta} Set, for $k \in \mN^*$ and $z \geq 0$,
\begin{align} 	\label{def-Pk}
P_k (z) = \left|e^{\Lambda_{2,k}(z)} - 1 \right|^2 \left(e^{2\Re (\Lambda_{1,k}(z))} - 1 \right),   
\end{align}
and 
\begin{align}	\label{def-Qk}
Q_k (z) = \Im \left\{ \left(e^{\Lambda_{2,k}(z)} - 1 \right) \left( e^{\overline{\Lambda_{1,k}(z)}} - 1 \right) \left(e^{2 i \Im (\Lambda_{1,k}(z))} - 1 \right) \right\}. 
\end{align}
Then, for {\bf even} $k \in \mN^*$ and $z \geq 0$,
\begin{align}
\Theta_k(z) &=  - \frac{2}{z^2 + \frac{(k \pi)^4}{4}} \left(\mathcal{A}(z) + \mathcal{B}(z) \right),
\end{align}
where
\begin{align}	\label{A(z)}
    \mathcal{A}(z) = \left(2 (k \pi)^2  + \sqrt{z^2 + \frac{(k \pi)^4}{4}}  \right)  P_k (z) \;\Re (\Lambda_{1,k}(z)), 
\end{align}
and
\begin{align}	\label{B(z)}
    \mathcal{B}(z) &= \left(2 (k \pi)^2  - \sqrt{z^2 + \frac{(k \pi)^4 }{4}} \right)  Q_k(z)  \; \Im (\Lambda_{1,k}(z)). 
\end{align} 
\end{proposition}
Here and in what follows, the real part and the imaginary part of a complex number $z$ are denoted by $\Re(z)$ and $\Im(z)$, respectively.

\medskip


The rest of this section is devoted to the proofs of \Cref{as-behavior} and \Cref{prop-Theta} and is organized as follows. We first establish various technical lemmas in \Cref{sect-Pre}. The proofs of \Cref{prop-Theta} and \Cref{as-behavior} are then given in \Cref{sect-prop-Theta} and \Cref{sect-as-behavior}, respectively.

\subsection{Preliminaries} \label{sect-Pre}

In the proofs of \Cref{as-behavior} and \Cref{prop-Theta}, we frequently use the following properties of $\Lambda_{1,k}(z)$ and $\Lambda_{2,k}(z)$, which are direct consequences of their definitions in \eqref{def-Lambda}.

\begin{lemma} \label{formulaforlambda}
Let $k \in \mathbb{N}^*$. We have
\begin{align} 	\label{Lamdaat0}  
\Lambda_{1,k}(0) = - \Lambda_{2,k}(0) = \displaystyle \frac{k \pi i}{\sqrt{2}},
\end{align}
and, for $z \in \mathbb{R} \setminus \{0\}$, 
\begin{align}
\label{sqroot}
    \Lambda_{1,k}(z) = - \Lambda_{2,k}(z) = \sqrt{\frac{1}{2}\sqrt{z^2 + \frac{(k \pi)^4}{4} } - \frac{(k \pi)^2}{4} } + i \; \mathrm{sign}(z) \; \sqrt{\frac{1}{2}\sqrt{z^2 + \frac{(k \pi)^4}{4}} + \frac{(k \pi)^2}{4}},
\end{align}  
where $\mathrm{sign}(z) = 1$  if $z > 0$ and $\mathrm{sign}(z) = -1$ if $z < 0$. Consequently, for  $z \in \mR$, 
 	\begin{align} \label{rexim}
 		\Re (\Lambda_{1,k}(z)) \; \Im (\Lambda_{1,k}(z)) = \frac{z}{2}, 
 	\end{align}
\begin{align}	\label{reim2}
	4 \Big(\Re(\Lambda_{1,k}(z)) \Big)^2 + (k \pi)^2 = 4 \Big( \Im (\Lambda_{1,k}(z)) \Big)^2 - (k \pi)^2 = 2 \sqrt{z^2 + \frac{(k \pi)^4}{4}}, 
\end{align} 	
and
\begin{align}	\label{behavereim}
\Re (\Lambda_{1,k} (z)) = \sqrt{\frac{z}{2}} + O\left(z^{-\frac{1}{2}} \right), \quad \Im (\Lambda_{1,k} (z)) = \sqrt{\frac{z}{2}} + O \left(z^{-\frac{1}{2}} \right) \mbox{ as }  z \to +\infty.
\end{align} 
\end{lemma}

We have the following representation of $\Theta_k$ defined by \eqref{Thetakformula}. 

\begin{lemma} \label{lem-Theta-I} Let $k \in \mN^*$  and let $z \ge 0$.  We have 
\begin{align}	\label{Thetak=sum}
   \Theta_k(z) = \sum_{j = 1}^5 I_j(z),
    \end{align} where  
 \begin{align}	\label{defI1}
   I_1(z) = \left|e^{\Lambda_{2,k}(z)} - 1 \right|^2 \int_0^1 e^{2x\Re (\Lambda_{1,k}(z))} \cos(k \pi x) dx, 
 \end{align}
 \begin{align}	\label{defI2}
        I_2(z) = \left|e^{\Lambda_{1,k}(z)} - 1 \right|^2 \int_{0}^1 e^{2x\Re (\Lambda_{2,k}(z))} \cos(k \pi x) dx, 
 \end{align}
 \begin{align}	\label{defI3}
        I_3(z) = 2\Re \left\{ \left(e^{\Lambda_{2,k}(z)} - 1 \right) \left(1 - e^{\overline{\Lambda_{1,k}(z)}} \right) \int_0^1  e^{\left(\Lambda_{1,k}(z) + \overline{\Lambda_{2,k}(z)} \right)x} \cos(k \pi x) dx \right\},
\end{align}
 \begin{align}    \label{defI4}
 I_4(z) = 2\Re \left\{ \left(e^{\Lambda_{2,k}(z)} - 1 \right)  \left(e^{\overline{\Lambda_{1,k}(z)}} - e^{\overline{\Lambda_{2,k}(z)}} \right) \int_0^1 e^{\Lambda_{1,k}(z) x} \cos(k \pi x) dx \right\},  
\end{align}
\begin{align}	\label{defI5}
        I_5(z) = 2\Re \left\{  \left(1 - e^{\Lambda_{1,k}(z)} \right) \left(e^{\overline{\Lambda_{1,k}(z)}} - e^{\overline{\Lambda_{2,k}(z)}} \right) \int_0^1 e^{\Lambda_{2,k}(z) x} \cos(k \pi x) dx \right\}. 
\end{align}
\end{lemma}

\begin{proof} 
By \eqref{Thetakformula}, we have
\begin{align} \label{lem-Theta-I-p1}
   \Theta_k(z) = \sum_{j = 1}^5 I_j(z) + \left|e^{\Lambda_{1,k}(z)} - e^{\Lambda_{2,k}(z)} \right|^2 \int_0^1 \cos(k \pi x) dx. 
    \end{align} 
    Since $\int_0^1 \cos(k \pi x) dx = 0$, the conclusion follows from \eqref{lem-Theta-I-p1}. 
\end{proof}

We next study $I_1, \cdots, I_5$ for $k$ even. Concerning $I_1$ and $I_2$ given in \eqref{defI1} and \eqref{defI2}, we have the following result. 

\begin{lemma} \label{lem-I1I2} Let $k \in \mN^*$ be {\bf even} and let $z \ge 0$. We have 
\begin{align}    \label{I2=I1} 
        I_1 (z) = I_2(z) = \frac{\Re(\Lambda_{1,k}(z))}{\sqrt{z^2 + \frac{(k \pi)^4}{4}}}  P_k(z), 
   \end{align}
where $P_k(z)$ is given in \eqref{def-Pk}.
\end{lemma}

\begin{proof}
We have
    \begin{align*}
        I_1(z)        &= \frac{1}{2} \left|e^{\Lambda_{2,k}(z)} - 1 \right|^2 \int_0^1 \left( e^{ \left(2\Re (\Lambda_{1,k}(z)) + k \pi i \right)x } + e^{ \left(2\Re(\Lambda_{1,k}(z)) - k \pi i \right) x} \right) dx \nonumber \\
        &= \frac{1}{2} \left|e^{\Lambda_{2,k}(z)} - 1 \right|^2 \left(\frac{e^{2\Re (\Lambda_{1,k}(z)) + k \pi i} - 1}{2\Re (\Lambda_{1,k}(z)) + k \pi i} + \frac{e^{2\Re(\Lambda_{1,k}(z)) - k \pi i} - 1}{2\Re (\Lambda_{1,k}(z)) - k \pi i}\right).
    \end{align*}
Using the fact that $e^{k \pi i} = 1$ since $k$ is even, we derive that 
  \begin{align*}
        I_1(z)        =  \frac{1}{2} \left|e^{\Lambda_{2,k}(z)} - 1 \right|^2 \left(e^{2\Re (\Lambda_{1,k}(z))} - 1 \right) \frac{4\Re (\Lambda_{1,k}(z))}{4 (\Re(\Lambda_{1,k}(z)))^2 + (k \pi)^2}. 
    \end{align*}

   Applying \eqref{reim2}, we get
    \begin{align}	\label{formula-I1}
	I_1(z) =      \frac{1}{\sqrt{z^2 + \frac{(k \pi)^4}{4}}}\left|e^{\Lambda_{2,k}(z)} - 1 \right|^2 \left(e^{2\Re (\Lambda_{1,k}(z))} - 1 \right)\Re(\Lambda_{1,k}(z)).
     \end{align}     
    Similarly,
   \begin{align}   	\label{formula-I2}
        I_2(z) = \frac{1}{\sqrt{z^2 + \frac{(k \pi)^4}{4}}}\left|e^{\Lambda_{1,k}(z)} - 1 \right|^2 \left(e^{2\Re (\Lambda_{2,k}(z))} - 1 \right)\Re(\Lambda_{2,k}(z)).
   \end{align}
Since, for $a, b \in \mR$, 
$$
|e^{a + i b} - 1|^2 \Big( e^{- 2 a} - 1 \Big) = - |e^{-(a + i b)} - 1|^2 \Big( e^{2 a} - 1 \Big), 
$$   
by applying \Cref{formulaforlambda}, we derive from \eqref{formula-I1} and \eqref{formula-I2} that $I_2(z) = I_1(z)$ and the conclusion follows from \eqref{formula-I1}. 
\end{proof}

Concerning $I_3$ given in  \eqref{defI3}, we have the following result. 

\begin{lemma} \label{lem-I3} Let $k \in \mN^*$ be {\bf even} and let $z \ge 0$. We have
\begin{align}\label{I3}  
I_3(z) = - \frac{2  \Im (\Lambda_{1,k}(z)) }{\sqrt{ z^2 + \frac{(k \pi)^4}{4} }} Q_k(z),  
\end{align}
where  $Q_k(z)$ is given in \eqref{def-Qk}.
\end{lemma}

\begin{proof}
We have    \begin{align*}
         \int_0^1  e^{\left(\Lambda_{1,k}(z) + \overline{\Lambda_{2,k}(z)} \right)x} \cos(k \pi x) dx & \mathop{=}^{\Cref{formulaforlambda}} \frac{1}{2} \int_0^1 e^{2i x \Im(\Lambda_{1,k}(z))} \left(e^{k \pi i x} + e^{- k \pi i x} \right)  dx \\
         &= \frac{1}{2} \left(\frac{e^{(2 \Im (\Lambda_{1,k}(z)) + k \pi)i} - 1}{(2 \Im (\Lambda_{1,k}(z)) + k \pi)i} + \frac{e^{(2 \Im (\Lambda_{1,k}(z)) - k \pi)i} - 1}{(2 \Im (\Lambda_{1,k}(z)) - k \pi)i} \right) 
    \end{align*}
since  $\Im(\Lambda_1(z)) \geq \frac{k \pi}{\sqrt{2}}$ by \Cref{formulaforlambda}. 
    
Since $e^{k \pi i} = 1$ for even $k$, we derive that 
     \begin{align*}
         \int_0^1  e^{\left(\Lambda_{1,k}(z) + \overline{\Lambda_{2,k}(z)} \right)x} \cos(k \pi x) dx 
  = \frac{-i}{2} \left(e^{2 i \Im (\Lambda_{1,k}(z)) } - 1 \right) \frac{4 \Im (\Lambda_{1,k}(z))}{4 (\Im (\Lambda_{1,k}(z)))^2 - (k \pi)^2}. 
    \end{align*}
Using \eqref{reim2}, we get
      \begin{align}	\label{I3-integral}
       \int_0^1  e^{\left(\Lambda_{1,k}(z) + \overline{\Lambda_{2,k}(z)} \right)x} \cos(k \pi x) dx &=  - \frac{i}{\sqrt{z^2 + \frac{(k \pi)^4}{4}}} \left(e^{2 i \Im (\Lambda_{1,k}(z)) } - 1 \right) \Im (\Lambda_{1,k}(z)). 
      \end{align}
 We deduce from \eqref{defI3} and \eqref{I3-integral} that
    $$
        I_3(z) 
        = - \frac{2}{\sqrt{z^2 + \frac{(k \pi)^4}{4}}} \; \Im (\Lambda_{1,k}(z)) \; \Im \left\{ \left(e^{\Lambda_{2,k}(z)} - 1 \right) \left( e^{\overline{\Lambda_{1,k}(z)}} - 1 \right) \left(e^{2 i \Im (\Lambda_{1,k}(z))} - 1 \right) \right\},
    $$
  which is \eqref{I3}. 
    \end{proof}
  
  Concerning $I_4$ and $I_5$ given in \eqref{defI4} and \eqref{defI5}, we have the following result. 
  
\begin{lemma} \label{lem-I4I5} Let $k \in \mN^*$ be {\bf even} and let $z \ge 0$. We have
\begin{multline}	\label{I4=I5}
- \frac{1}{2} \left(z^2 + \frac{(k \pi)^4}{4} \right) I_4(z)  = - \frac{1}{2} \left(z^2 + \frac{(k \pi)^4}{4} \right) I_5(z) 
\\[6pt]
=  P_k (z) \left( \frac{(k \pi)^2}{2}\Re (\Lambda_{1,k}(z)) + z \Im (\Lambda_{1,k}(z)) \right) \\
 -  Q_k(z) \left(z\Re (\Lambda_{1,k}(z))  - \frac{(k \pi)^2}{2} \Im (\Lambda_{1,k}(z)) \right) ,
\end{multline}
where $P_k(z)$ and $Q_k(z)$ are given in \eqref{def-Pk} and \eqref{def-Qk}. 
\end{lemma}
    
\begin{proof} Note that $\Lambda_{1,k}(z) + k \pi i \neq 0$ and $\Lambda_{1,k}(z) - k \pi i \neq 0$ by \eqref{Lamdaat0} and \eqref{sqroot}. Thus,
    \begin{multline*}
        \int_0^1 e^{\Lambda_{1,k}(z) x} \cos(k \pi x)  dx = \frac{1}{2} \int_0^1 \left(e^{(\Lambda_{1,k}(z) + k \pi i) x} + e^{(\Lambda_{1,k}(z) - k \pi i) x} \right) dx  \\
        = \frac{1}{2} \left(\frac{e^{\Lambda_{1,k}(z) + k \pi i} - 1}{\Lambda_{1,k}(z) + k \pi i} + \frac{e^{\Lambda_{1,k}(z) - k \pi i} - 1}{\Lambda_{1,k}(z) - k \pi i} \right) \mathop{=}^{e^{k \pi i} = 1} \frac{\left(e^{\Lambda_{1,k}(z)} - 1 \right) \Lambda_{1,k}(z) }{(\Lambda_{1,k}(z))^2 + (k \pi)^2}. 
    \end{multline*}
    Since, by the definition of $\Lambda_{1,k}(z)$, 
    \begin{align*}
   |\Lambda_{1,k}(z)|^4 = \left|iz - \frac{(k \pi)^2}{2} \right|^2 = z^2 + \frac{(k \pi)^4}{4}, \quad      (\Lambda_{1,k}(z))^2 + (k \pi)^2 = iz + \frac{(k \pi)^2}{2} = - \overline{(\Lambda_{1,k}(z))^2},  
    \end{align*}
it follows that    
\begin{align}	\label{I4-integral}
    	\int_0^1 e^{\Lambda_{1,k}(z) x} \cos(k \pi x) dx = - \frac{1}{z^2 + \frac{(k \pi)^4}{4}} \left(e^{\Lambda_{1,k}(z)} - 1 \right) \left(\Lambda_{1,k}(z) \right)^3.  
    \end{align}
    We deduce from \eqref{defI4} and \eqref{I4-integral} that 
    \begin{align}
        I_4(z) = - \frac{ 2}{z^2 + \frac{(k \pi)^4}{4}}\Re \left\{  \left(e^{\Lambda_{1,k}(z)} - 1 \right) \left(e^{\Lambda_{2,k}(z)} - 1 \right) \left(e^{\overline{\Lambda_{1,k}(z)}} - e^{\overline{\Lambda_{2,k}(z)}} \right) \left(\Lambda_{1,k}(z) \right)^3 \right\}. 	\label{I4-step1}
    \end{align}
    Similarly,
    \begin{multline}  	\label{I5-step1}
        I_5(z) = -\frac{ 2}{z^2 + \frac{(k \pi)^4}{4}}\Re \left\{ \left(e^{\Lambda_{1,k}(z)} - 1 \right)  \left(e^{\Lambda_{2,k}(z)} - 1 \right)\left(e^{\overline{\Lambda_{2,k}(z)}} - e^{\overline{\Lambda_{1,k}(z)}} \right) \left(\Lambda_{2,k}(z) \right)^3 \right\}  \\ 
        = I_4(z).
    \end{multline}  

  	Note that 
    \begin{align}
       \Re \left\{ \left(\Lambda_{1,k}(z) \right)^3 \right\} &= \Re \left\{\left(-\frac{(k \pi)^2}{2} + iz \right) \Lambda_{1,k}(z) \right\} = - \frac{(k \pi)^2}{2}\Re (\Lambda_{1,k}(z)) - z \Im (\Lambda_{1,k}(z)), \label{I4-step3.1}\\
        \Im \left\{ \left(\Lambda_{1,k}(z) \right)^3 \right\} &=  \Im \left\{\left(-\frac{(k \pi)^2}{2} + iz \right) \Lambda_{1,k}(z) \right\} = z\Re (\Lambda_{1,k}(z)) - \frac{(k \pi)^2}{2} \Im (\Lambda_{1,k}(z)). 	\label{I4-step3.2}
    \end{align}
        
We claim that    
\begin{align}	\label{I4-step4.1} 
  \Re \left\{ \left(e^{\Lambda_{1,k}(z)} - 1 \right) \left(e^{\Lambda_{2,k}(z)} - 1 \right) \left(e^{\overline{\Lambda_{1,k}(z)}} - e^{\overline{\Lambda_{2,k}(z)}} \right) \right\} = - P_k(z), 
\end{align}     
and	
\begin{align}	\label{I4-step4.2}
\Im \left\{ \left(e^{\Lambda_{1,k}(z)} - 1 \right) \left(e^{\Lambda_{2,k}(z)} - 1 \right) \left(e^{\overline{\Lambda_{1,k}(z)}} - e^{\overline{\Lambda_{2,k}(z)}} \right)  \right\} = Q_k(z).
\end{align}
Admitting this claim, we obtain \eqref{I4=I5} from \eqref{I4-step1}, \eqref{I5-step1}, \eqref{I4-step3.1}, \eqref{I4-step3.2}, \eqref{I4-step4.1}, and \eqref{I4-step4.2}.

\medskip

    It remains to prove \eqref{I4-step4.1} and \eqref{I4-step4.2}. We have, since $e^{\Lambda_{1,k}(z) + \Lambda_{2,k}(z)} = e^0 = 1$, 
    \begin{multline*}
         \left(e^{\Lambda_{1,k}(z)} - 1 \right) \left(e^{\Lambda_{2,k}(z)} - 1 \right) \left(e^{\overline{\Lambda_{1,k}(z)}} - e^{\overline{\Lambda_{2,k}(z)}} \right) \\
        = 2 \left(e^{\overline{\Lambda_{1,k}(z)}} - e^{\overline{\Lambda_{2,k}(z)}} \right) - \left(e^{\Lambda_{1,k}(z) + \overline{\Lambda_{1,k}(z)}} - e^{\Lambda_{2,k}(z) + \overline{\Lambda_{2,k}(z)}} + e^{\Lambda_{2,k}(z) + \overline{\Lambda_{1,k}(z)}} - e^{\Lambda_{1,k}(z) + \overline{\Lambda_{2,k}(z)}}\right). 
    \end{multline*}
Since
    \begin{align*}
    	e^{\Lambda_{1,k}(z) + \overline{\Lambda_{1,k}(z)}} - e^{\Lambda_{2,k}(z) + \overline{\Lambda_{2,k}(z)}}  \in \mR, \quad  e^{\Lambda_{2,k}(z) + \overline{\Lambda_{1,k}(z)}} - e^{\Lambda_{1,k}(z) + \overline{\Lambda_{2,k}(z)}} \in i \mR, 
    \end{align*}
it follows that
\begin{multline}\label{I4I5-p1}
        \Re \left\{ \left(e^{\Lambda_{1,k}(z)} - 1 \right) \left(e^{\Lambda_{2,k}(z)} - 1 \right) \left(e^{\overline{\Lambda_{1,k}(z)}} - e^{\overline{\Lambda_{2,k}(z)}} \right) \right\} \\ = e^{\Lambda_{1,k}(z)} - e^{\Lambda_{2,k}(z)} + e^{\overline{\Lambda_{1,k}(z)}} - e^{\overline{\Lambda_{2,k}(z)}}  - e^{\Lambda_{1,k}(z) + \overline{\Lambda_{1,k}(z)}} + e^{\Lambda_{2,k}(z) + \overline{\Lambda_{2,k}(z)}},  
    \end{multline}
    and
    \begin{multline}\label{I4I5-p2}
          \Im \left\{ \left(e^{\Lambda_{1,k}(z)} - 1 \right) \left(e^{\Lambda_{2,k}(z)} - 1 \right) \left(e^{\overline{\Lambda_{1,k}(z)}} - e^{\overline{\Lambda_{2,k}(z)}} \right)  \right\} \\[6pt]
          = \Im \left( e^{\overline{\Lambda_{1,k}(z)}} - e^{\overline{\Lambda_{2,k}(z)}} - e^{\Lambda_{1,k}(z)} + e^{\Lambda_{2,k}(z)} - e^{\Lambda_{2,k}(z) + \overline{\Lambda_{1,k}(z)}} + e^{\Lambda_{1,k}(z) + \overline{\Lambda_{2,k}(z)}} \right). 
    \end{multline}  
    
Applying \Cref{lem-Pk} below, we obtain \eqref{I4-step4.1} from \eqref{I4I5-p1}. Applying \Cref{lem-Qk} below, we reach \eqref{I4-step4.2} from \eqref{I4I5-p2}. 

\medskip 
The proof is complete.  
\end{proof}

In the proof of \Cref{lem-I4I5}, we have used the following two results. 

\begin{lemma} \label{lem-Pk} Let $k \in \mN^*$ and let $z \geq 0$. We have
\be \label{lem-Pk-cl1}
P_k(z) = - e^{\overline{\Lambda_{1, k}(z)}} - e^{\Lambda_{1, k}(z)} + e^{2 \Re (\Lambda_{1, k}(z))}
- e^{\Lambda_{2, k}(z) + \overline{\Lambda_{2, k}(z)}} +  e^{\Lambda_{2, k}(z)} + e^{\overline{\Lambda_{2, k}(z)}}. 
\ee 
Consequently, 
\begin{align}\label{lem-Pk-cl2}
P_k(z) =
e^{2\Re (\Lambda_{1,k}(z))} - e^{-2\Re (\Lambda_{1,k}(z))}  - 2 \Re \left(e^{ \Lambda_{1,k}(z)} - e^{-\Lambda_{1,k}(z)} \right),
\end{align}
and
\be 	\label{lem-Pk-cl3}
P_k(z) \geq
e^{2\Re (\Lambda_{1,k}(z))} - e^{-2\Re (\Lambda_{1,k}(z))}  - 2 e^{ \Re(\Lambda_{1,k}(z))} + 2 e^{-\Re (\Lambda_{1,k}(z))}.
\ee
\end{lemma}

\begin{proof} We have, by the definition of $P_k$, 
\begin{multline}\label{Pk-p1}
P_k(z) = e^{\Lambda_{2, k}(z) + \overline{\Lambda_{2, k}(z)} + 2 \Re (\Lambda_{1, k}(z))} - e^{\Lambda_{2, k}(z)  + 2 \Re (\Lambda_{1, k}(z))} -e^{\overline{\Lambda_{2, k}(z)} + 2 \Re (\Lambda_{1, k}(z)) } + e^{2 \Re \Lambda_{1, k}(z)}  \\[6pt]
- \left(e^{\Lambda_{2, k}(z) + \overline{\Lambda_{2, k}(z)}} - e^{\Lambda_{2, k}(z)} -e^{\overline{\Lambda_{2, k}(z)}} + 1\right). 
\end{multline}
Since, by \Cref{formulaforlambda}, 
$$
\Lambda_{2, k}(z) + \overline{\Lambda_{2, k}(z)} + 2 \Re (\Lambda_{1, k}(z)) = 0, \; \;  \Lambda_{2, k}(z)  + 2 \Re (\Lambda_{1, k}(z)) = \overline{\Lambda_{1, k}(z)}, \; \;   \overline{\Lambda_{2, k}(z)} + 2 \Re \Lambda_{1, k}(z) = \Lambda_{1, k}(z),  
$$
assertions \eqref{lem-Pk-cl1} and  \eqref{lem-Pk-cl2} follow from \eqref{Pk-p1}.

To derive \eqref{lem-Pk-cl3} from \eqref{lem-Pk-cl2}, one just notes that, since $\Re(\Lambda_{1,k}(z)) \geq 0$ for $z \geq 0$,
\begin{align*}
\Re \left(e^{ \Lambda_{1,k}(z)} - e^{-\Lambda_{1,k}(z)} \right) =  \left(	e^{ \Re(\Lambda_{1,k}(z))} - e^{- \Re(\Lambda_{1,k}(z))} \right) \cos(\Im(\Lambda_{1,k}(z))) \leq  	e^{ \Re(\Lambda_{1,k}(z))} - e^{- \Re(\Lambda_{1,k}(z))}.
\end{align*}
 The proof is complete. 
\end{proof}

\begin{lemma} \label{lem-Qk} Let $k \in \mN^*$ and let $z \ge 0$. We have 
\begin{multline} \label{lem-Qk-cl1}
\left(e^{\Lambda_{2,k}(z)} - 1 \right) \left(e^{\overline{\Lambda_{1,k}(z)}} - 1 \right) \left(e^{2i \Im (\Lambda_{1,k}(z)) } - 1 \right) \\[6pt]
= - e^{\overline{\Lambda_{2, k}(z)}} - e^{\Lambda_{1, k}(z)} +  e^{\Lambda_{1,k}(z) + \overline{\Lambda_{2,k}(z)}} 
- e^{\Lambda_{2,k}(z) + \overline{\Lambda_{1,k}(z)}} + e^{\Lambda_{2,k}(z)} + e^{\overline{\Lambda_{1,k}(z)}}. 
\end{multline}
Consequently,  
\begin{align} \label{lem-Qk-cl2} 
 Q_k(z) = - 2 \left(e^{\Re (\Lambda_{1,k}(z))} + e^{-\Re (\Lambda_{1,k}(z))} - 2 \cos (\Im (\Lambda_{1,k}(z)) ) \right)  \sin (\Im (\Lambda_{1,k}(z))).
\end{align}
\end{lemma}

\begin{proof} We have 
\begin{multline}\label{Qk-p1}
\left(e^{\Lambda_{2,k}(z)} - 1 \right) \left(e^{\overline{\Lambda_{1,k}(z)}} - 1 \right) \left(e^{2i \Im (\Lambda_{1,k}(z)) } - 1 \right) \\[6pt]
= e^{\Lambda_{2,k}(z) + \overline{\Lambda_{1,k}(z)} + 2i \Im (\Lambda_{1,k}(z)) } - e^{\Lambda_{2,k}(z) + 2i \Im (\Lambda_{1,k}(z)) } - e^{\overline{\Lambda_{1,k}(z)} + 2i \Im (\Lambda_{1,k}(z)) }  + e^{2i \Im (\Lambda_{1,k}(z)) } \\[6pt]
- \left( e^{\Lambda_{2,k}(z) + \overline{\Lambda_{1,k}(z)}} - e^{\Lambda_{2,k}(z)} - e^{\overline{\Lambda_{1,k}(z)}}  + 1  \right).
\end{multline}
Since, by \Cref{formulaforlambda}, 
$$
\Lambda_{2,k}(z) + \overline{\Lambda_{1,k}(z)} + 2i \Im (\Lambda_{1,k}(z)) =0, \quad \Lambda_{2,k}(z) + 2i \Im (\Lambda_{1,k}(z)) = \overline{\Lambda_{2, k}(z)},
$$ 
$$
\overline{\Lambda_{1,k}(z)} + 2i \Im (\Lambda_{1,k}(z)) = \Lambda_{1, k}(z), \quad 2i \Im (\Lambda_{1,k}(z))  = \Lambda_{1,k}(z) + \overline{\Lambda_{2,k}(z)},
$$
assertion \eqref{lem-Qk-cl1} follows from \eqref{Qk-p1}. 

It follows from \eqref{lem-Qk-cl1} and \Cref{formulaforlambda} that 
\begin{align} 
 Q_k(z) = - 2 \Im \left(e^{\Lambda_{1,k}(z)} - e^{- \Lambda_{1,k}(z)} + e^{-2 i \Im(\Lambda_{1,k}(z))} \right),
\end{align}
which implies \eqref{lem-Qk-cl2}. 
The proof is complete. 
\end{proof}

\subsection{Proof of Proposition \ref{prop-Theta}} \label{sect-prop-Theta} Let $k \in \mN^*$ be even and let $z \geq 0$. Applying \Cref{lem-Theta-I}, \Cref{lem-I1I2}, \Cref{lem-I3}, and \Cref{lem-I4I5}, we have 
    \begin{align*}	
        \Theta_k(z) &= 2I_1(z) + I_3(z) + 2I_4(z) = - \frac{2}{z^2 + \frac{(k \pi)^4}{4}} \left(\widetilde{\mathcal{A}}(z) + \widetilde{\mathcal{B}}(z) \right),
    \end{align*}  
    where
    \begin{align*}
    \widetilde{\mathcal{A}}(z) = \left(- \sqrt{z^2 + \frac{(k \pi)^4}{4} } \;\Re(\Lambda_{1,k}(z)) + (k \pi)^2\Re (\Lambda_{1,k}(z)) + 2z \Im (\Lambda_{1,k}(z)) \right) P_k(z),
    \end{align*}
    and
    \begin{align*}
    \widetilde{\mathcal{B}}(z) = \left(\sqrt{z^2 + \frac{(k \pi)^4}{4} } \; \Im (\Lambda_{1,k}(z)) - 2 z\Re(\Lambda_{1,k}(z)) + (k \pi)^2 \Im (\Lambda_{1,k}(z)) \right) Q_k(z).
    \end{align*}
    
    It remains to show that $\widetilde{\mathcal{A}}(z) = \mathcal{A}(z)$ and $ \widetilde{\mathcal{B}}(z) =  \mathcal{B}(z)$. Indeed, we have, by \Cref{formulaforlambda},  
   \begin{multline*}
        \widetilde{\mathcal{A}}(z) \mathop{=}^{\eqref{rexim}}  \left(- \sqrt{z^2 + \frac{(k \pi)^4}{4}} + (k \pi)^2  + 4 \left(\Im (\Lambda_{1,k}(z)) \right)^2 \right) P_k(z) \;\Re (\Lambda_{1,k}(z))\\
        \mathop{=}^{\eqref{reim2}} \left(2(k \pi)^2 + \sqrt{z^2 + \frac{(k \pi)^4}{4}}\right) P_k(z) \;\Re (\Lambda_{1,k}(z)) = \mathcal{A}(z),
    \end{multline*}
    and 
    \begin{multline*}
        \widetilde{\mathcal{B}}(z)  \mathop{=}^{\eqref{rexim}}  \left( \sqrt{z^2 + \frac{(k \pi)^4}{4}} - 4 (\Re (\Lambda_{1,k}(z)))^2 + (k \pi)^2 \right) Q_k(z) \; \Im (\Lambda_{1,k}(z))  \\
        \mathop{=}^{\eqref{reim2}} \left(2(k \pi)^2 - \sqrt{z^2 + \frac{(k \pi)^4}{4}} \right) Q_k(z) \; \Im (\Lambda_{1,k}(z)) = \mathcal{B}(z).
    \end{multline*}
The proof of \Cref{prop-Theta} is complete.
\qed

\subsection{Proof of Proposition \ref{as-behavior}} \label{sect-as-behavior}

This section, which consists of four subsections, is devoted to the proof of \Cref{as-behavior}.
\subsubsection{Proof of Assertion 1)} \label{sect-as-behavior-p1}
It follows from \eqref{sqroot} that
\begin{align}	\label{L-z=CLz}
    \Lambda_{i, k} \left(-z \right) = \overline{\Lambda_{i, k}\left(z \right)}   \mbox{ for } z \in (0, +\infty) \text{ and } i \in \{1, 2\}. 
\end{align}
Combining \eqref{L-z=CLz} with \eqref{Omegakformula} and \eqref{Thetakformula}, we deduce that $\Omega_k$ is even. It also follows from \eqref{Lamdaat0} and \eqref{sqroot} that $\Lambda_{1,k}$ and $\Lambda_{2,k}$ are continuous on $\mR \setminus \{0\}$ and right continuous at $0$. In view of  \eqref{Omegakformula} and \eqref{Thetakformula}, this implies that $\Omega_k$ is continuous on $\mR \setminus \{0\}$ and right continuous at $0$. Since $\Omega_k$ is even and right continuous at $0$, $\Omega_k$ is continuous at $0$. Thus, $\Omega_k$ is continuous on $\mR$.\qed

\subsubsection{Proof of Assertion 2)} \label{sect-as-behavior-p2}
 By \eqref{Omegakformula} and \Cref{prop-Theta}, for $z > 0$, 
\begin{align}
   \Omega_k(z)  = \frac{-2 \left(\mathcal{A}(z) + \mathcal{B}(z) \right)}{\left(z^2 + \frac{(k \pi)^4}{4} \right)^2 \left|e^{\Lambda_{1,k}(z)} - e^{\Lambda_{2,k}(z)} \right|^2},	\label{asym1}
\end{align}
where $\mathcal{A}(z)$ and $\mathcal{B}(z)$ are given in \eqref{A(z)} and \eqref{B(z)}, respectively. 

We have, for large positive $z$, 
\begin{align*}
	\mathcal{A}(z) &=  \left(2(k \pi)^2 + \sqrt{z^2 + \frac{(k \pi)^4}{4} }\right) \left|e^{\Lambda_{2,k}(z)} - 1 \right|^2 \left(e^{2\Re (\Lambda_{1,k}(z))} - 1 \right) \Re (\Lambda_{1,k}(z)) \\[6pt]
	& \mathop{=}^{\eqref{behavereim}}  \left(z + O(1) \right) \left(1 + O \left(e^{-\Re (\Lambda_{1,k}(z))} \right) \right) \left(e^{2\Re (\Lambda_{1,k}(z))} - 1 \right) \left(\sqrt{\frac{z}{2}} + O \left(z^{-\frac{1}{2}} \right) \right) \\[6pt]
	&= \frac{1}{\sqrt{2}} e^{2\Re (\Lambda_{1,k}(z))} z^{\frac{3}{2}} + O\left(e^{2\Re (\Lambda_{1,k}(z))} z^{\frac{1}{2}} \right), 
\end{align*}
and 
\begin{align*}
	\mathcal{B}(z) &= \left(2(k \pi)^2 - \sqrt{z^2 + \frac{(k \pi)^4}{4} } \right) \Im \left\{ \left(e^{\Lambda_{2,k}(z)} - 1 \right) \left( e^{\overline{\Lambda_{1,k}(z)}} - 1 \right) \left(e^{2 i \Im (\Lambda_{1,k}(z))} - 1 \right) \right\} \Im (\Lambda_{1,k}(z)) \\[6pt]
	& \mathop{=}^{\eqref{behavereim}}  O(z) \; O\left(e^{\Re(\Lambda_{1,k}(z))} \right) \; O \left(z^{\frac{1}{2}} \right)  = O \left(e^{\Re (\Lambda_{1,k}(z))} z^{\frac{3}{2}} \right).
\end{align*}
Hence,
\begin{align}
	\mathcal{A}(z) + \mathcal{B}(z) &=  \frac{1}{\sqrt{2}} e^{2\Re (\Lambda_{1,k}(z))} z^{\frac{3}{2}} + O\left(e^{2\Re (\Lambda_{1,k}(z))} z^{\frac{1}{2}} \right).	\label{asym2}
\end{align} 
On the other hand, 
\begin{align}
\frac{1}{\left(z^2 + \frac{(k \pi)^4}{4} \right)^2 \left|e^{\Lambda_{1,k}(z)} - e^{\Lambda_{2,k}(z)} \right|^2} & \mathop{=}^{\eqref{behavereim}}  \left(\frac{1}{z^4} + O \left(z^{-6} \right) \right) \left(\frac{1}{e^{2\Re (\Lambda_{1,k}(z))}} + O \left(\frac{1}{e^{4\Re(\Lambda_{1,k}(z))}} \right) \right) \nonumber \\
&= \frac{1}{z^4 e^{2\Re (\Lambda_{1,k}(z))}} + O \left(\frac{1}{z^6 e^{2\Re (\Lambda_{1,k}(z))}} \right).	\label{asym3}
\end{align}
We deduce from \eqref{asym1}, \eqref{asym2}, and \eqref{asym3} that 
\begin{align*}
	\Omega_k(z) &= -2 \left(\frac{1}{\sqrt{2}} e^{2\Re (\Lambda_{1,k}(z))} z^{\frac{3}{2}} + O \left(e^{2\Re (\Lambda_{1,k}(z))} z^{\frac{1}{2}} \right) \right) \left(\frac{1}{z^4 e^{2\Re (\Lambda_{1,k}(z))}} + O \left(\frac{1}{z^6 e^{2\Re (\Lambda_{1,k}(z))}} \right) \right),
\end{align*}
which implies
\begin{align}	\label{Omegak-asym}
	\Omega_k(z) = -\sqrt{2} z^{-\frac{5}{2}} + O(z^{-\frac{7}{2}}).
\end{align}  
Assertion $2)$ follows from \eqref{Omegak-asym} and the evenness of $\Omega_k$.\qed

\subsubsection{Proof of Assertion 3)} \label{sect-as-behavior-p3}
As $\Omega_k$ is even, it suffices to show that $\Omega_k(z) < 0$ for
$z \in \left(\frac{3(k \pi)^2}{2 \sqrt{7}}, +\infty \right)$. In this proof, we assume that 
$$
 z \in \left(\frac{3(k \pi)^2}{2 \sqrt{7}}, +\infty \right). 
$$

By Proposition \ref{prop-Theta}, it suffices to show that 
    \begin{multline}   \label{shownegative}
        - Q_k(z)  \left(2(k \pi)^2 - \sqrt{z^2 + \frac{(k \pi)^4}{4}} \right) \Im (\Lambda_{1,k}(z)) \\
        < P_k(z)  \left( 2(k \pi)^2 + \sqrt{z^2 + \frac{(k \pi)^4}{4} } \right)\Re (\Lambda_{1,k}(z)). 
    \end{multline}

We claim that
    \begin{align}   
         \label{signc}
     \left|2(k \pi)^2 - \sqrt{z^2 + \frac{(k \pi)^4}{4} } \right|   \Im (\Lambda_{1,k}(z))  <  \left( 2(k \pi)^2 + \sqrt{z^2 + \frac{(k \pi)^4}{4} }\right)\Re (\Lambda_{1,k}(z)),
    \end{align}
    and
    \begin{align}	\label{signc2}
    |Q_k(z)| < P_k(z).
    \end{align}

Admitting this claim, we derive \eqref{shownegative} from \eqref{signc} and \eqref{signc2} after noting that $\Im(\Lambda_{1,k}(z)) > 0$.

\medskip

It remains to prove \eqref{signc} and \eqref{signc2}. 

\medskip 
We first deal with \eqref{signc}. Set
\begin{align}	\label{defa}
a_k = \displaystyle \sqrt{z^2 + \frac{(k \pi)^4}{4}}. 
\end{align}
Then the condition $z > \frac{3(k \pi)^2}{2\sqrt{7}}$ implies 
\begin{align}	\label{a > ...}
	a_k  > \frac{2(k \pi)^2}{\sqrt{7}}.
\end{align}
Since 
$$
(2 -t)^2 (t/2 + 1/4) < (2 + t)^2 (t/2 - 1/4) \mbox{ for } t > 2/ \sqrt{7}, 
$$
it follows from \eqref{a > ...} that
\begin{align}\label{signc-1}
	 \left( 2(k \pi)^2 - a_k \right)^2 \left(\frac{a_k}{2} + \frac{(k \pi)^2}{4} \right) <  \left(  2(k \pi)^2 + a_k \right)^2 \left( \frac{a_k}{2} - \frac{(k \pi)^2}{4} \right).
\end{align}
Assertion \eqref{signc} now follows from \eqref{signc-1} and \Cref{formulaforlambda}. 

\medskip    
    
We next establish \eqref{signc2}.  We have, by \Cref{lem-Qk}, 
\begin{align} 
 Q_k(z)  = - 2 \left(e^{\Re (\Lambda_{1,k}(z))} + e^{-\Re (\Lambda_{1,k}(z))} - 2 \cos (\Im (\Lambda_{1,k}(z)) ) \right)  \sin (\Im (\Lambda_{1,k}(z))).	\label{Qk2form}
\end{align}
Using \eqref{Qk2form} and noting that $e^{-\Re(\Lambda_{1,k}(z))} \leq 1$, we get
    \begin{align}	\label{sign-case2-1}
         |Q_k(z)|   \leq 2 \left(e^{\Re (\Lambda_{1,k}(z))} + 3 \right).
    \end{align}
   We have 
    \begin{multline}
       P_k(z) \mathop{\geq}^{\Cref{lem-Pk}} e^{2\Re (\Lambda_{1,k}(z))}  - e^{-2\Re (\Lambda_{1,k}(z))} - 2 e^{ \Re (\Lambda_{1,k}(z))}  + 2 e^{-\Re (\Lambda_{1,k}(z))} \\
          \mathop{>}^{\Re (\Lambda_{1, k}(z)) > 0} e^{2\Re (\Lambda_{1,k}(z))} - 2 e^{ \Re (\Lambda_{1,k}(z))} - 1.		\label{sign-case2-2}
    \end{multline}
    By \eqref{sign-case2-1} and \eqref{sign-case2-2}, it suffices to show the following inequality to deduce \eqref{signc2}:
  \begin{align}\label{sign-case2-4}		
        2 \left(e^{\Re (\Lambda_{1,k}(z))} + 3 \right) \leq e^{2\Re (\Lambda_{1,k}(z))} - 2 e^{\Re (\Lambda_{1,k}(z))} - 1. 
    \end{align}    
  To this end, set $t = e^{\Re(\Lambda_{1,k}(z))}$. Then \eqref{sign-case2-4} becomes
\begin{align*}
	2(t + 3) \leq t^2 - 2t - 1, \quad \text{ i.e., } \quad  t^2 - 4t - 7 \geq 0,
\end{align*}    
which is equivalent to (since $t > 0$)
\begin{align}	\label{sign-case2-3}
t \geq 2 + \sqrt{11}.
\end{align}
This indeed holds since 
\begin{align*}
       \Re (\Lambda_{1,k}(z)) \mathop{=}^{\eqref{sqroot}} \sqrt{\frac{a_k}{2} -\frac{(k \pi)^2}{4}  } \mathop{>}^{\eqref{a > ...}} k \pi \sqrt{\frac{1}{\sqrt{7}} - \frac{1}{4}} \geq 2 \pi \sqrt{\frac{1}{\sqrt{7}} - \frac{1}{4}} > \ln (2 + \sqrt{11}). 
    \end{align*}
    
    The proof is complete.  \qed
  
\subsubsection{Proof of Assertion 4)}
By Assertion $3)$ and the evenness of $\Omega_{k_0}$, it suffices to show that
\be \label{assertion4-p1}
\Omega_{k_0}(z) < 0 \mbox{ for } k_0 \in \{2, 10\} \text{ and } z \in \left[0, \frac{3(k_0 \pi)^2}{2 \sqrt{7}}\right]. 
\ee
Since $P_{k_0}(z)\geq 0$, $\Re (\Lambda_{1,k_0}(z)) \geq 0$, and $\Im(\Lambda_{1, k_0}(z)) > 0$, assertion \eqref{assertion4-p1} follows from the following fact that we will prove below:
\begin{align}
 \label{shownegative-2}
         Q_{k_0}(z)  \left(2(k_0 \pi)^2 - \sqrt{z^2 + \frac{(k_0 \pi)^4}{4}} \right)  > 0.
\end{align} 

We claim that
\begin{align}
 e^{\Re (\Lambda_{1,k_0}(z))} + e^{-\Re (\Lambda_{1,k_0}(z))} - 2 \cos (\Im (\Lambda_{1,k_0}(z)) ) > 0, \label{Step1-p1}
\end{align}
 \begin{align}   \label{z < ...}
        2({k_0} \pi)^2 - \sqrt{z^2 + \frac{({k_0} \pi)^4}{4} } > 0,
    \end{align} 
    and
     \begin{align}   \label{sin<0}
        \sin (\Im (\Lambda_{1,k_0}(z))) < 0. 
    \end{align} 
Admitting this claim, we obtain \eqref{shownegative-2} from \Cref{lem-Qk}. 

In the rest of the proof, we establish \eqref{Step1-p1}, \eqref{z < ...}, and \eqref{sin<0}. We begin with \eqref{Step1-p1}. We have
\begin{multline}\label{Step1-p1-p1}
e^{\Re(\Lambda_{1,k_0}(z))} + e^{- \Re(\Lambda_{1,k_0}(z))} - 2 \cos (\Im (\Lambda_{1,k_0}(z)) )  \\
  = e^{-\Re(\Lambda_{1,k_0}(z))} \left(e^{\Re(\Lambda_{1,k_0}(z))} - 1 \right)^2 + 2 \left[ 1 -  \cos(\Im (\Lambda_{1,k_0}(z))) \right].
\end{multline}
Since $1 - \cos(\Im(\Lambda_{1,k_0}(0))) = 1 - \cos(\frac{{k_0} \pi}{\sqrt{2}}) > 0$ and $e^{\Re(\Lambda_{1,k_0}(z))} - 1 > 0$ for $z > 0$, we obtain \eqref{Step1-p1} from \eqref{Step1-p1-p1}.

Set
\begin{align}
a = a_{k_0} = \sqrt{z^2 + \frac{(k_0 \pi)^4}{4}}.
\end{align}
Since $ 0  \leq z \leq \frac{3({k_0} \pi)^2}{2 \sqrt{7}}$, it follows that
\begin{align}	\label{rangea}
			 	\frac{({k_0} \pi)^2}{2} \leq a \leq \frac{2({k_0} \pi)^2}{\sqrt{7}}, 
\end{align}
which yields \eqref{z < ...}.   

It remains to prove \eqref{sin<0}. We have, 
    \begin{align}	\label{rangeImLambda}
        \Im (\Lambda_{1,k_0}(z)) \mathop{=}^{\Cref{formulaforlambda}} \sqrt{\frac{a}{2} + \frac{({k_0} \pi)^2}{4} } \mathop{\in}^{\eqref{rangea}} \left[\frac{{k_0} \pi}{\sqrt{2}}, {k_0} \pi \sqrt{\frac{1}{4} + \frac{1}{\sqrt{7}}} \right].
    \end{align}
    When ${k_0} = 2$, \eqref{rangeImLambda} implies
\begin{align*}
	\pi < \sqrt{2} \pi \leq \Im(\Lambda_{1,2}(z)) \leq 2 \pi \sqrt{\frac{1}{4} + \frac{1}{\sqrt{7}}} < 2 \pi.
\end{align*}        
      When ${k_0} = 10$, \eqref{rangeImLambda} implies
        \begin{align*}
        7 \pi < \frac{10}{\sqrt{2}} \pi \leq \Im(\Lambda_{1,10}(z)) \leq 10 \pi \sqrt{\frac{1}{4} + \frac{1}{\sqrt{7}}} < 8 \pi.
        \end{align*}
In both cases, \eqref{sin<0} is satisfied. 

The proof is complete. \qed

\section{A quadratic obstruction for the local controllability - Proof of (\ref{key-1+2})} \label{sect-obstruction}

In this section, we prove the following result which is the key ingredient in the proof of \Cref{Mainresult}.

\begin{proposition} \label{2obstruction}
Let $T > 0$. Let $u_1 \in L^1(\mR)$ be real-valued and let $$y_1 \in C([0, +\infty), L^2(0, 1)) \cap L^2_{loc}((0, +\infty), H^1(0, 1))$$ be the unique solution of \eqref{our1storder-***}. 
Suppose that 
$$
\mathrm{supp} \; u_1 \subset [0, T] \quad \mbox{ and } \quad  y_1(T, \cdot) = 0.
$$ 
Then there exists a positive constant $C$ depending only on $T$ such that, for $k_0 \in \{2, 10\}$,
\begin{align}	\label{key1+2}
 \int_0^T \int_0^1 y_1^2(t, x) \phi_{k_0, x}(t, x) dx dt \leq  - C \|u_1\|^2_{[H^{\frac{5}{4}}(0, T)]^*}, 
\end{align}
where $\phi_{k_0}$ is given in \eqref{def-phik}.
\end{proposition} 

\begin{proof} Let $k_0 \in \{2,10\}$.
By \Cref{lem-key-1} and \Cref{as-behavior},
    \begin{align}   \label{<=Ck}
     \int_0^T \int_0^1 y_1^2(t, x) \phi_{k_0, x}(t, x) dx dt &\leq - C  \int_{\mathbb{R}} \frac{\left|\hu_1\left(z + i \frac{({k_0} \pi)^2}{2} \right) \right|^2}{\left(1 + |z|^2 \right)^{\frac{5}{4}}} dz. 
    \end{align}
Here and throughout the proof, $C$ denotes a positive constant that depends only on $T$, and the value of $C$ varies line by line.  
   
 Consider the function  $v: \mathbb{R} \to \mathbb{R}$ defined by 
        \begin{align}	\label{def-v}
 v(t) = e^{\frac{1}{2} k_0^2 \pi^2 t}  u_1(t)  \quad \mbox{ for } t \in \mR. 
        \end{align}
    Then 
    \begin{align*}
    		\hat{v}(z) = \hu_1 \left(z + \frac{1}{2} i k_0^2 \pi^2 \right) \quad \forall z \in \mC.
    \end{align*} 
 It follows from \eqref{<=Ck}  that 
\begin{align}   \label{vandu}
        \int_0^T \int_0^1 y_1^2(t, x) \phi_{k_0, x}(t, x) dx dt  \leq -C \|v\|^2_{H^{-\frac{5}{4}}(\mR)}.
    \end{align}
Since $u_1 \in L^1(\mR)$ with $\mathrm{supp} \; u_1 \subset [0,T]$, we have 
\begin{align}	\label{-5/4norms}
	\| u_1\|_{[H^{\frac{5}{4}}(0, T)]^*} \leq C \|u_1\|_{H^{-\frac{5}{4}}(\mR)}.
\end{align}

We claim that
\begin{align}	\label{usimv-1}
\|u_1 \|_{H^{-\frac{5}{4}}(\mathbb{R})} \leq C \|v\|_{H^{-\frac{5}{4}}(\mathbb{R})}.
\end{align}
Admitting this claim, we obtain \eqref{key1+2} from \eqref{vandu}, \eqref{-5/4norms}, and \eqref{usimv-1}.

In the rest of the proof,  we establish \eqref{usimv-1}. It suffices to show that for every $\varphi \in C_c^\infty(\mR)$,
\begin{align} \label{obstruction2-p0}
	\left| \int_{\mR} u_1(t) \varphi(t) dt \right| \leq C \|v \|_{H^{-\frac{5}{4}}(\mR)} \|\varphi \|_{H^{\frac{5}{4}}(\mR)}.
\end{align}
To this end, fix $\chi \in C^{\infty}_c(\mathbb{R})$ such that $\mathrm{supp} \; \chi \subset [-T, 2T]$ and $\chi = 1$ in $[0, T]$. Since $\mathrm{ supp } \; v \subset [0, T]$ and $\chi = 1$ in $[0, T]$, it follows from \eqref{def-v} that 
\begin{align*}
 \int_\mR u_1(t) \varphi(t)  dt = \int_\mR v(t) e^{-\frac{1}{2} k_0^2 \pi^2 t}   \varphi(t)  dt  
 = \int_\mR v(t) \chi(t) e^{-\frac{1}{2} k_0^2 \pi^2 t}   \varphi(t)  dt.
\end{align*}
We derive that 
\begin{align}	\label{2obstruction-p2}
	\left|\int_{\mathbb{R}} u_1(t) \varphi(t) dt \right| \leq  \|v\|_{H^{-\frac{5}{4}}(\mR)}\|\chi e^{-\frac{1}{2} k_0^2 \pi^2 \cdot} \varphi \|_{H^{\frac{5}{4}}(\mR)}.
\end{align}
Since $H^{\frac{5}{4}}(\mR)$ is an algebra, it follows that 
\be \label{2obstruction-p3}
 \|\chi e^{-\frac{1}{2} k_0^2 \pi^2 \cdot} \varphi \|_{H^{\frac{5}{4}}(\mR)} \le C \|\chi e^{-\frac{1}{2} k_0^2 \pi^2 \cdot} \|_{H^{\frac{5}{4}}(\mR)}\|\varphi \|_{H^{\frac{5}{4}}(\mR)}  \le C  \|\varphi \|_{H^{\frac{5}{4}}(\mR)}.
\ee
Estimate \eqref{obstruction2-p0} is a consequence of \eqref{2obstruction-p2} and \eqref{2obstruction-p3}. 

\medskip 
The proof of \Cref{2obstruction} is complete.
\end{proof}

\begin{remark} \rm One can also prove that  
\begin{align}
  \|v\|_{H^{-\frac{5}{4}}(\mathbb{R})} \leq C \|u_1 \|_{H^{-\frac{5}{4}}(\mathbb{R})}. 
\end{align}
In other words, the $H^{-\frac{5}{4}}(\mR)$ norms of $v$ and $u_1$ are equivalent. 
\end{remark}

\section{Estimates for the linearized system}	\label{subsec-2ineq}
In this section, we establish estimates for the linearized system of \eqref{ViscousBurgers1D} which is 
\begin{align}
	\label{y1y2}
            \begin{cases}
                y_{t}(t, x) - y_{xx} (t, x) = u(t) \quad &\text{ in } (0, T) \times (0, 1), \\[6pt]
                y(t, 0) = y(t, 1) = 0 &\text{ in } (0, T), \\[6pt]
                y(0, x) = 0 &\text{ in } (0, 1).
            \end{cases}
\end{align}
Set, for $ u \in L^1(0, T)$, 
\begin{align} \label{U'=u}
	U(t) = \int_0^t u(s) ds \quad \text{ for } t \in [0, T].
\end{align}

Here is the main result of this section. 

\begin{proposition}	\label{pro-y1} Let $T > 0$, $u \in L^1(0, T)$, and let $y \in Y_T$ be the unique solution of \eqref{y1y2}. There exists a positive constant $C$ depending only on $T$ such that
\begin{align}	\label{E-1-y1-1}
	\|y \|_{L^2((0, T), H^1(0, 1))} \leq C \|u\|_{[H^{\frac{3}{4}}(0, T)]^*}, 
\end{align}
and
\begin{align}
           \|y - U\|_{L^2((0, T) \times (0, 1))} \leq C \|u\|_{[H^{\frac{5}{4}}(0, T)]^*},  \label{E-y1-2-1}  
\end{align}
where $U$ is defined by \eqref{U'=u}. As a consequence of \eqref{E-y1-2-1},
\begin{align}
	\|y\|_{L^2((0, T) \times (0, 1))} \leq C \|u\|_{[H^1(0, T)]^*}.
\end{align}
\end{proposition}

\begin{remark} \rm \label{rem-y1} 
The standard estimate for solutions of \eqref{y1y2} is (see, e.g., \Cref{ap-heat-fx})
\begin{align}
	\|y\|_{C([0, T], L^2(0, 1))} + \|y\|_{L^2((0, T), H^1(0, 1))} \leq C \|u\|_{L^1(0, T)}.
\end{align}
Proposition \ref{pro-y1} asserts that  the upper bound for $\|y \|_{L^2((0, T), H^1(0, 1))}$ can be improved from $\|u\|_{L^1(0, T)}$ to $\|u\|_{[H^{\frac{3}{4}}(0, T)]^*}$ (note that $\|u\|_{[H^s(0, T)]^*} \leq C \|u\|_{L^1(0, T)}$ for all $s > \frac{1}{2}$ by the Sobolev embedding of $H^s(0, T)$ into $L^\infty(0, T)$ for $s > \frac{1}{2}$). 

\end{remark}

\begin{remark} \rm Estimate \eqref{E-y1-2-1} was previously obtained by Marbach \cite[Lemma 22]{Marbach18} whose proof is different from ours. It is unclear to us whether his approach can be used to establish \eqref{E-1-y1-1}. 
 \end{remark}

The following lemma plays a key role in the proof of \Cref{pro-y1}.

\begin{lemma} \label{lem-spectral}
Let $\gamma \in (-\frac{1}{2}, \frac{3}{2})$, and let $V \in L^1(\mR)$ be such that $\mathrm{supp} \; V \subset [0, +\infty)$. Set 
\begin{align}	\label{lem-spectral-defak}
a_k (t) =  k^{\gamma} \int_0^t e^{- \pi^2 k^2 (t-s)} V(s) \, ds \quad \mbox{ for } k \in \mN^* \text{ and } t \geq 0.  
\end{align}
There exists a positive constant $C$ depending only on $\gamma$ such that
\begin{align}	\label{lem-spectral-ineq}
\sum_{k  \geq 1} \int_0^{+\infty} |a_k(t)|^2 \, dt \leq C \| V \|_{H^{ \frac{\gamma}{2} - \frac{3}{4} }(\mR)}^2.
\end{align}
\end{lemma}

In the rest of this section, we give the proof of \Cref{lem-spectral} and \Cref{pro-y1} in \Cref{subsec-prooflemspectral} and \Cref{subsec-proofproy1}, respectively.

\subsection{Proof of Lemma \ref{lem-spectral}	}	\label{subsec-prooflemspectral}

We begin this section with the following key ingredient for the proof of \Cref{lem-spectral}.

\begin{lemma}	\label{lem-spectral-key-g}
	Let $\gamma \in (-\frac{1}{2}, \frac{3}{2})$. Set 
	\begin{align*}
		g_{\gamma}(s) = \sum_{k \geq 1} k^{2 \gamma - 2} e^{- \pi^2 k^2 |s|}  \quad \text{ for } s \in \mR.
	\end{align*}
	The following properties hold:
	\begin{enumerate}
	\item[1)] $g_{\gamma} \in L^1(\mR)$.
	\item[2)] For every $\xi \in \mR$, 
\begin{align} \label{lem-spectral-key-g-2}	
	\hat{g}_{\gamma}(\xi) = \sqrt{2}\pi^{\frac{3}{2}} \sum_{k \geq 1} \frac{k^{2 \gamma}}{\pi^4 k^4 + \xi^2}.
\end{align}
	\item[3)] There exists a positive constant $C$ depending only on $\gamma$ such that for every $\xi \in \mR$,
	\begin{align}	\label{lem-spectral-key-g-3}
		|\hat{g}_{\gamma}(\xi)| \leq C \left(1 + |\xi|^2 \right)^{\frac{\gamma}{2} - \frac{3}{4}}.
	\end{align}
	\end{enumerate}
\end{lemma}

\begin{proof}
By Fatou's lemma, we get
\begin{align*}
	\int_{\mR} |g_{\gamma}(s)| ds 	\leq \liminf_{N \to \infty} \sum_{k = 1}^N  k^{2 \gamma - 2} \int_{\mR} e^{- \pi^2 k^2 |s|} ds = 2 \pi^{-2} \sum_{k \geq 1} k^{2\gamma - 4} \mathop{<}^{\gamma < \frac{3}{2}} +\infty. 
\end{align*}
In other words, $g_\gamma \in L^1(\mR)$. 

We have, for $\xi \in \mR$,
\begin{multline*}
	\hat{g}_{\gamma}(\xi) = \frac{1}{\sqrt{2 \pi}} \sum_{k \geq 1} k^{2\gamma - 2} \int_{\mR} e^{- i \xi s} e^{- \pi^2 k^2 |s|} ds \\
	 = \frac{1}{\sqrt{2 \pi}} \sum_{k \geq 1} k^{2 \gamma - 2} \left(\frac{1}{\pi^2 k^2 + i \xi } + \frac{1}{\pi^2 k^2 - i \xi}  \right) = \sqrt{2}\pi^{\frac{3}{2}} \sum_{k \geq 1} \frac{k^{2 \gamma}}{\pi^4 k^4 + \xi^2},  
\end{multline*} 
which is \eqref{lem-spectral-key-g-2}. 

Let $\xi \in \mR $ be such that $|\xi| \geq 1$ and let $k_0 \in \mN^*$ be such that $k_0 \leq |\xi|^{1/2} < k_0 + 1$. We have
\begin{align} \label{lem-spec-p2}
\sum_{k=1}^{k_0} \frac{k^{2 \gamma}}{k^4 + \xi^2} \leq  \sum_{k=1}^{k_0} \frac{k^{2 \gamma}}{\xi^2}   \mathop{\leq}^{\gamma > -\frac{1}{2}} C_{\gamma} k_0^{2 \gamma + 1} |\xi|^{-2} \leq C_{\gamma} |\xi|^{\gamma - \frac{3}{2}}, 
\end{align}
and
\begin{align}  \label{lem-spec-p3}
\sum_{k = k_0 + 1}^{+\infty} \frac{k^{2 \gamma}}{k^4 + \xi^2} \leq    \sum_{k = k_0 + 1}^{+\infty} k^{2 \gamma-4} \mathop{\leq}^{\gamma < \frac{3}{2}} C_{\gamma} k_0^{2\gamma - 3} \leq C_{\gamma} |\xi|^{\gamma - \frac{3}{2}}. 
\end{align}
Combining \eqref{lem-spectral-key-g-2} with \eqref{lem-spec-p2} and \eqref{lem-spec-p3} and noting that $\hat{g}_{\gamma}$ is bounded in $\mR$ (since $g_{\gamma} \in L^1(\mR)$), we obtain \eqref{lem-spectral-key-g-3}. The proof is complete. 
\end{proof}

We now prove \Cref{lem-spectral}.

\begin{proof}[Proof of \Cref{lem-spectral}] By a standard approximation argument involving the Fatou lemma, one can assume that $V \in C^{\infty}(\mR)$ as well. This condition will be assumed from now on.

\medskip
Let $k \in \mN^*$. We have
\begin{align*}
	\int_0^{+\infty} |a_k(t)|^2 dt &= k^{2 \gamma} \int_0^{+\infty} \left(\int_0^t e^{- \pi^2 k^2 (t - s_1)} V(s_1) ds_1 \right) \left(\int_0^t e^{- \pi^2 k^2 (t - s_2)} V(s_2) ds_2 \right) dt  \nonumber \\[6pt]
	&= k^{2 \gamma} \int_0^{+\infty} \left(\int_0^t \int_0^t e^{- \pi^2 k^2 (2t - s_1 - s_2)} V(s_1) V(s_2) ds_1 ds_2 \right) dt.  
\end{align*}
Applying Fubini's theorem, we obtain 
\begin{align}
	\int_0^{+\infty} |a_k(t)|^2 dt &= k^{2 \gamma} \iint_{(0, +\infty)^2}  V(s_1) V(s_2) \left(\int_{\max\{s_1, s_2\}}^{+\infty} e^{- \pi^2 k^2 (2t - s_1 - s_2)} dt \right) ds_1 ds_2 \nonumber \\[6pt]
	&= \frac{1}{2 \pi^2} \int_{\mR} \int_{\mR} V(s_1) V(s_2) k^{2\gamma - 2} e^{- \pi^2 k^2 |s_1 - s_2|} ds_1 ds_2, 		\label{lem-spec-ak}
\end{align}
since $\mbox{supp } V \subset [0, + \infty)$.

Set 
\begin{align*}		g_{\gamma}(s) = \sum_{k \geq 1} k^{2 \gamma - 2} e^{- \pi^2 k^2 |s|} \quad \text{ for } s \in \mR.
\end{align*}
It follows from \eqref{lem-spec-ak} that 
\begin{align} \label{lem-spec-ak-p1}
	2 \pi^2 \sum_{k \geq 1} \int_0^{+\infty} |a_k(t)|^2 dt  = \int_{\mR} V(s) (V \ast g_{\gamma})(s) ds. 
\end{align}
Applying Parseval's formula and Lemma \ref{lem-spectral-key-g}, we get
\begin{multline}		\label{lem-spec-As1-2}	
\int_{\mR} V(s) (V \ast g_{\gamma})(s) ds = \int_{\mR}  \overline{\hat{V}(\xi)} \hat{V}(\xi) \hat{g}_{\gamma}(\xi) d\xi = \int_{\mR} |\hat{V}(\xi)|^2 \hat{g}_{\gamma}(\xi) d\xi \\
\leq C_{\gamma} \int_{\mR} |\hat{V}(\xi)|^2 \left(1 + |\xi|^2 \right)^{\frac{\gamma}{2} - \frac{3}{4}} d \xi \leq C_{\gamma} \| V\|_{H^{\frac{\gamma}{2} - \frac{3}{4}}(\mR)}^2.  
\end{multline}
 
Combining \eqref{lem-spec-ak-p1} and \eqref{lem-spec-As1-2}, we obtain \eqref{lem-spectral-ineq}. The proof is complete.
\end{proof}

\subsection{Proof of Proposition \ref{pro-y1}}	\label{subsec-proofproy1}
\textit{Step 1: Proof of \eqref{E-1-y1-1}.}  Represent  $y$ under the form
\begin{align}	\label{lem-y1-U-1}
y(t, x) = \sum_{k \ge 1}a_k (t) e_k(x) \quad \mbox{ in } (0, T) \times (0,1),  
\end{align}
where $e_k$ is defined by \eqref{def-ek}.  Since $y$ solves \eqref{y1y2}, it follows that for every $k \in \mN^*$,
\begin{align}	\label{lem-y1-U-2}
a_k'(t) +   k^2 \pi^2 a_k(t) = u_k (t) \quad \text{ in } (0, T) \quad \mbox{ and } \quad  a_k (0) = 0,
\end{align}
where
\begin{align}	\label{lem-y1-U-3}
u_k (t) = \sqrt{2}  u(t) \int_0^1 \sin (k \pi x ) \, dx = \omega_k u(t) \quad \text{ in } (0, T),
\end{align}
where
\begin{align} \label{lem-y1-wk}
\omega_k = \frac{\sqrt{2} \left( 1 - (-1)^k \right)}{k \pi}.  
\end{align}
We deduce from \eqref{lem-y1-U-2} and \eqref{lem-y1-U-3} that for every $k \in \mN^*$, 
\begin{align}	\label{lem-y1-U-4}
a_k(t) &=  \omega_k \int_0^t u(s) e^{-  \pi^2 k^2 (t - s)} ds \quad \text{ in } (0, T).
\end{align}

Combining  \eqref{lem-y1-U-1} and \eqref{lem-y1-U-4}, we obtain 
\begin{align}	\label{lem-y1x-U-1}
y_{x}(t, x) = \sum_{k \ge 1} k \pi a_k (t) f_k(x) \quad \mbox{ in } (0, T) \times (0,1), 
\end{align}
where 
$$
f_k(x) = \frac{1}{k \pi} e_k'(x) = \sqrt{2} \cos (k \pi x) \mbox{ in } [0, 1]. 
$$

Denote by $u_e$ the extension of $u$ by 0 outside $[0, T]$, i.e., 
\begin{align} \label{pro-y1-ue}
	u_e(t) = \begin{cases}
	u(t)		 \quad &\text{ if } t \in [0, T], 	\\
	0		\quad &\text{ if } t \in \mR \setminus [0, T], 
	\end{cases}
\end{align}
and set
\begin{align} \label{pro-y1-tb-p1}
a_{k, e} (t) = \omega_k \int_0^t u_e(s) e^{-  \pi^2 k^2 (t - s)} ds  \quad \text{ for } k \geq 1, \;   t \ge 0.
\end{align}
Then
\be\label{pro-y1-tb}
a_{k, e} (t) = a_k (t) \mbox{ in } (0, T) \mbox{ for } k \ge 1.
\ee
We thus obtain 
\begin{align*}
	\|y_{x} \|_{L^2((0, T) \times (0, 1))}^2 \mathop{=}^{\eqref{lem-y1x-U-1}} \int_0^T \sum_{k \geq 1} |\pi k a_k(t)|^2 dt 
	\mathop{\leq}^{\eqref{pro-y1-tb}}	 \sum_{k \geq 1} \int_0^{+\infty} |\pi k a_{k, e} (t)|^2 dt, 
\end{align*}
which yields 
\begin{align}\label{pro-y1-p1}
	\|y_{x} \|_{L^2((0, T) \times (0, 1))}^2  \le C \sum_{k \geq 1} \int_0^{+\infty} |b_k(t)|^2 dt \mbox{ where } b_k(t) = \int_0^t u_e(s) e^{-  \pi^2 k^2 (t - s)} ds. 
\end{align}

Applying  Lemma \ref{lem-spectral} with $V = u_e$, we derive from  \eqref{pro-y1-tb} that
\begin{align}\label{pro-y1-p2}
	\sum_{k \geq 1} \int_0^{+\infty} |b_{k} (t)|^2 dt \le  C \|u_e \|_{H^{-\frac{3}{4}}(\mR)}^2 \mathop{\le}^{\eqref{pro-y1-ue}} C \|u\|_{[H^{\frac{3}{4}}(0, T)]^*}^2. 
\end{align}
Combining \eqref{pro-y1-p1} and  \eqref{pro-y1-p2}, we obtain
\begin{align}	\label{E-1-y1-1'}
	\|y_x \|_{L^2((0, T) \times (0, 1))} \leq C \|u\|_{[H^{\frac{3}{4}}(0, T)]^*}, 
\end{align}
which yields \eqref{E-1-y1-1} since $y(t, 0) = y(t, 1) = 0$ for $t \in (0, T)$. 

\medskip

\textit{Step 2: Proof of \eqref{E-y1-2-1}.} Represent $U$ under the form 
\begin{align*}	U(t) = \sum_{k \geq 1} U_k(t) e_k(x)	\quad \text{ in } (0, T) \times (0, 1), 
\end{align*}
where 
\begin{align} \label{lem-y1-def-Uk}	U_k(t) = \sqrt{2} U(t) \int_0^1  \sin(k \pi x) dx  = \omega_k U(t)
\end{align}
with $\omega_k$ being defined by \eqref{lem-y1-wk}. Then 
\begin{align}
	(y - U)(t) = \sum_{k \geq 1} \Big(a_k(t) - U_k(t) \Big) e_k(x) \quad \text{ in } (0, T) \times (0, 1). 
\end{align}
Note that, for $k \geq 1$ and $t \geq 0$,
$$
	 a_k(t) - U_k(t) \mathop{=}^{\eqref{lem-y1-U-4},\eqref{lem-y1-def-Uk}} \omega_k \left(\int_0^t u(s) e^{- \pi^2 k^2 (t - s)} ds - U(t) \right), 
$$
and, by integration by parts,  
\begin{multline*}
	\int_0^t u(s) e^{- \pi^2 k^2 (t - s)} ds - U(t) = - k^2 \pi^2 \int_0^t U(s) e^{- \pi^2 k^2 (t - s)} ds  \\[6pt]
	= - k^2 \pi^2 c_k(t) \mbox{ where } c_k(t) = \int_0^t U(s) e^{- \pi^2 k^2 (t - s)} ds. 
	\end{multline*}
It follows that, for $k \ge 1$ and $t \ge 0$,
$$
a_k(t)  - U(t) = - \omega_k k^2 \pi^2 c_k(t).  
$$
We thus obtain 
\begin{align*}
\|y - U\|_{L^{2}((0, T) \times (0, 1))}^2 &= \sum_{k \geq 1} \int_0^T  |\omega_k k^2 \pi^2 c_k(t)|^2 dt \mathop{\le}^{\eqref{lem-y1-wk}}  \sum_{k \geq 1} \int_0^T  k^2 \pi^2 | c_k(t)|^2 dt. 
\end{align*}
Applying Lemma \ref{lem-spectral} with $\gamma = 1$, we obtain 
\begin{align*}
	 \| y - U \|_{L^2((0, T) \times (0, 1))}  \leq C \|U\|_{[H^{\frac{1}{4}}(0, T)]^*}. 
\end{align*}
Since $ \|U\|_{[H^{\frac{1}{4}}(0, T)]^*} \le C \|u\|_{[H^{\frac{5}{4}}(0, T)]^*}$ by \Cref{lem-Ueandu} below, assertion \eqref{E-y1-2-1} follows. 

\medskip 

The proof of  \Cref{pro-y1} is complete. \qed

\medskip 
In the last part of the proof of \Cref{pro-y1}, we used the following standard simple result whose proof is given in \Cref{sect-lem-Ueandu} for the convenience of the readers. 

\begin{lemma} \label{lem-Ueandu}Let $T>0$, $s \ge 1$ and $u \in L^1(0, T)$. There exists a positive constant $C$ depending only on $T$ such that 
\begin{align}	\label{lem-Ueandu-conclusion}
	 \|U\|_{[H^{s - 1}(0, T)]^*} \le C \| u\|_{[H^s(0, T)]^*},
\end{align}
where
\begin{align*}
	U(t) = \int_0^t u(s) ds \quad \text{ for } t \in [0, T].
\end{align*}
\end{lemma}

\section{Proof of Theorem \ref{Mainresult}}	\label{provemainsresult}

The main result of this section, which implies Theorem \ref{Mainresult}, is as follows. 
  
    \begin{theorem}     \label{coercivity->obstruction}
        Let $T > 0$ and let $k_0 \in \{2, 10\}$. There exists $\varepsilon_0 > 0$ depending only on $T$ such that for every $0 < \varepsilon < \varepsilon_0$ and for every $u \in L^1(0, T)$ with $\|u\|_{[H^{\frac{3}{4}}(0, T)]^*} < \varepsilon_0$, we have
            \begin{align*}
                y\left(T, \cdot\right) \neq 0,
            \end{align*}
            where $y \in Y_{T}$ is the unique solution of the system
        \begin{align}   \label{Burgersdetailed}
            \begin{cases}
                y_t(t, x) - y_{xx} (t, x) + yy_x (t, x) = u(t) \quad &\text{ in } (0, T) \times (0, 1), \\
                y(t, 0) = y(t, 1) = 0 &\text{ in } (0, T), \\
                y(0, x) =  - \varepsilon \sin (k_0 \pi x) &\text{ in } (0, 1).
            \end{cases}
        \end{align}        
    \end{theorem}

The proof of \Cref{coercivity->obstruction} is based on the power series expansion. To this end, we study 
 \begin{align}\label{M-y1}
\begin{cases}
  y_{1,t} (t, x) - y_{1, xx} (t, x) = u(t)  \quad &\text{ in } (0, T) \times (0, 1), \\
                y_1(t, 0) = y_1(t, 1) = 0 &\text{ in } (0, T), \\
                y_1(0, x) = 0 &\text{ in } (0, 1), 
\end{cases}
\end{align}
and 
\begin{align} \label{M-y2}
             \begin{cases}
                y_{2,t} - y_{2, xx} = -y_1 y_{1,x}  \quad &\text{ in } (0, T) \times (0, 1), \\
                y_2(t, 0) = y_2(t, 1) = 0 &\text{ in } (0, T), \\
                y_2(0, x) = 0 &\text{ in } (0, 1). 
            \end{cases}   
\end{align}  
The study of $y_1$ is already given in \Cref{pro-y1}. The rest of this section, which consists of two subsections, is organized as follows.  In the first subsection, we establish estimates for the systems $y_2$, $\delta y$, and $y$, where  $y$ is the solution of 
   \begin{align}
        \begin{cases}   \label{M-y}
            y_t - y_{xx} + yy_x = u(t) \quad &\text{ in } (0, T) \times (0, 1), \\
            y(t, 0) = y(t, 1) = 0 &\text{ in } (0, T), \\
            y(0, x) = y_0(x) &\text{ in } (0, 1),
        \end{cases}
        \end{align}
        and 
        \be
        \mbox{$\delta y = y - y_1 - y_2$.}
        \ee
        The proof of  Theorem \ref{coercivity->obstruction} is then given in the second subsection.

\subsection{Some useful lemmas} \label{sub-sec-1st2ndorder}

We begin with estimates for $y_2$. 

\begin{lemma} \label{lem-y2small} Let $T > 0$ and let $u \in L^1(0, T)$. Let $y_1, y_2 \in Y_T$ be the unique solutions of 
\eqref{M-y1} and \eqref{M-y2}, respectively. Then
\begin{align}
   \|y_2\|_{Y_T}  \leq C \|u\|_{[{H^{\frac{3}{4}}(0, T)]^*}}^2, \label{E-y2-2}
\end{align}
and
\begin{align}	\label{E-y2-2'}
	\|y_2\|_{L^2((0, T) \times (0, 1))} \leq C \|u\|_{[H^{\frac{5}{4}} (0, T)]^*} \|u\|_{[H^{\frac{3}{4}}(0, T)]^*}, 
\end{align}
for some positive constant $C$ depending only on $T$. 
\end{lemma}

\begin{remark} \rm \label{rem-y1-2} It is worth noting that $y_1$ satisfies all the estimates given in \Cref{pro-y1}.  
\end{remark}

\begin{proof}
	We have, by \Cref{ap-heat-fx} in \Cref{wp-bg},  
	$$	
			\|y_2 \|_{Y_T} \le C \|y_1 y_{1, x} \|_{L^1((0, T), L^2(0, 1))}. 
	$$
Since
$$	
			 \|y_1 y_{1, x} \|_{L^1((0, T), L^2(0, 1))} 	\leq C \|y_1 \|_{L^2((0, T), H^1(0, 1))}^2 		\mathop{\leq}^{\Cref{pro-y1}} C \| u\|^2_{[H^{\frac{3}{4}}(0, T)]^*}, 
$$
assertion \eqref{E-y2-2} follows.

\medskip
	
	We next establish \eqref{E-y2-2'}. We have, by \Cref{lem-B2} in \Cref{wp-bg}, 
	\begin{align}	\label{L2y2-0}
		\|y_2 \|_{L^2((0, T) \times (0, 1))} \le C \|y_1 y_{1, x} \|_{ \left(C \left([0, T], H^1_0(0, 1) \right) \right)^*}.
	\end{align}
Note that
	\begin{align}	\label{L2y2-1}
		y_1 y_{1, x} =  \partial_x\Big((y_1 - U) y_1  \Big) - (y_1 - U) y_{1, x} \quad \mbox{ where } \quad U(t) = \int_0^t u(s) \, ds. 
	\end{align}
We have
\begin{multline}	\label{L2y2-2}
	\left \| \partial_x \Big( (y_1 - U) y_1 \Big) \right \|_{\left(C \left([0, T], H^1_0(0, 1) \right)  \right)^*} \leq C \| (y_1 - U) y_1 \|_{L^1((0, T), L^2(0, 1))} \\[6pt]
	\le 
		 C \|y_1 - U \|_{L^2((0, T) \times (0, 1))} \|y_1 \|_{L^2((0, T), H^1(0, 1))}.
\end{multline}
We also have 
	\begin{multline}	\label{L2y2-3}
		\| (y_1 - U) y_{1, x}\|_{\left(C \left([0, T], H^1_0(0, 1) \right)  \right)^*} \leq C \| (y_1 - U) y_{1, x}\|_{L^1((0, T) \times (0, 1))} \\[6pt]
		 \leq C \|y_1 - U\|_{L^2((0, T) \times (0, 1))} \|y_{1, x} \|_{L^2((0, T) \times (0, 1))}.
	\end{multline}
	Combining \eqref{L2y2-1}, \eqref{L2y2-2}, and \eqref{L2y2-3}, we obtain
	\begin{align}	\label{yy1XT'}
		\| y_1 y_{1, x} \|_{\left(C \left([0, T], H^1_0(0, 1) \right)  \right)^*} \leq C \|y_1 - U\|_{L^2((0, T) \times (0, 1))} \|y_1 \|_{L^2((0, T), H^1(0, 1))}.
	\end{align}
	Applying \Cref{pro-y1} to $y_1$, we derive from \eqref{yy1XT'} that
\begin{align}
		 \|y_1 y_{1, x} \|_{\left(C \left([0, T], H^1_0(0, 1) \right)  \right)^*}   \le C \|u\|_{[H^{\frac{5}{4}}(0, T)]^*} \|u\|_{[H^{\frac{3}{4}}(0, T)]^*}.	\label{L2y2-4}
	\end{align}			
Combining \eqref{L2y2-0} and \eqref{L2y2-4}, we obtain \eqref{E-y2-2'}. The proof is complete. 
\end{proof}


We next establish several estimates for $\delta y$.

\begin{lemma}\label{lem-dysmall}  Let $T > 0$, $u \in L^1(0, T)$, and $y_0 \in L^2(0, 1)$. Let $y_1$, $y_2$, and $y$ be the unique solutions in $Y_T$ of \eqref{M-y1}, \eqref{M-y2}, and \eqref{M-y}, respectively. 
Set
\begin{align}
	\delta y = y - y_1 - y_2 \mbox{ in } (0, T) \times (0, 1).
\end{align}
There exist positive constants $\varepsilon$ and $C$ depending only on $T$ such that, if 
\begin{align}
\|y_0\|_{L^2(0, 1)} + \|u\|_{[H^{\frac{3}{4}}(0, T)]^*} \le \varepsilon, 
\end{align}
then
\begin{align}	\label{lem-deltay-1}
	\|\delta y\|_{Y_T} \leq C \|u\|_{[H^{\frac{3}{4}}(0, T)]^*}^3 + C \|y_0\|_{L^2(0, 1)},
\end{align}
and
\begin{align}	\label{lem-deltay-2}
	\| \delta y\|_{L^2((0, T) \times (0, 1))} \leq  C \| u\|_{[H^{\frac{5}{4}}(0, T)]^*}^{\frac{3}{2}} \| u\|_{[H^{\frac{3}{4}}(0, T)]^*}^{\frac{3}{2}} + C \|y_0\|_{L^2(0, 1)}.
\end{align}
\end{lemma}

\begin{proof}
Note that $\delta y \in Y_T$ and $\delta y$ satisfies
\begin{align}	\label{sys-deltay}
	\begin{cases}
		(\delta y)_t - (\delta y)_{xx}  = -yy_x + y_1 y_{1,x} \quad &\text{ in } (0, T) \times (0, 1), \\
		(\delta y)(t, 0) = (\delta y)(t, 1) = 0 &\text{ in } (0, T), \\
		(\delta y)(0, x) = y_0(x) &\text{ in } (0, 1).
	\end{cases}
\end{align}	
Since
\begin{multline}  \label{yyx-y1y1x}
	-y y_x + y_1 y_{1, x}  = -\frac{1}{2} \partial_x(y^2 - y_1^2) = -\frac{1}{2} \partial_x  \Big( (y-y_1) (y + y_1) \Big) \\
	= -\frac{1}{2} \partial_x \Big((y_2 + \delta y)(2y_1 + y_2 + \delta y) \Big), 
\end{multline}
it follows that 
\begin{align}	\label{delta-gx}
(\delta y)_t - (\delta y)_{xx} + (\delta y) (\delta y)_x = - g_x \quad \mbox{ in } (0, T) \times (0, 1),
\end{align}
where
\begin{align} \label{yyx-y1y1x-2}
g = y_1 y_2 + y_2^2/2  + (y_1 + y_2) \delta y. 
\end{align}
We derive from \eqref{delta-gx} and  \Cref{lem-ap1} in \Cref{wp-bg} that 
	\begin{align}
		\|\delta y\|_{Y_T} &\leq C \left\| g_x \right\|_{L^1((0, T), L^2(0, 1))} + C \|y_0\|_{L^2(0, 1)}.	\label{deltay-YT}
	\end{align}	
Here and throughout the proof, $C$ denotes a positive constant depending only on $T$, and the value of $C$ varies line by line. 

We have
	\begin{multline*}	
	\left\| g_x \right\|_{L^1((0, T), L^2(0, 1))} \leq \|g\|_{L^1((0, T), H^1(0, 1))} \\[6pt] \leq C \left(\|y_1\|_{L^2((0, T), H^1(0, 1))} + \|y_2\|_{L^2((0, T), H^1(0, 1))} \right) \left(\|y_2\|_{L^2((0, T), H^1(0, 1))}  + \|\delta y\|_{L^2((0, T), H^1(0, 1))} \right). 
	\end{multline*}
	This  implies 
	\begin{align}	\label{deltay-YT-2}
	\left\| g_x \right\|_{L^1((0, T), L^2(0, 1))} 		\leq C \left( \|y_1\|_{L^2((0, T), H^1(0, 1))} + \|y_2\|_{Y_T}\right) \left( \|y_2\|_{Y_T} + \|\delta y\|_{Y_T} \right). 
	\end{align}
Applying \Cref{pro-y1} to $y_1$ and using \Cref{lem-y2small}, we have 
\begin{align} \label{deltay-YT-12}
\|y_1\|_{L^2((0, T), H^1(0, 1))} \le C \|u\|_{[H^{\frac{3}{4}}(0, T)]^*} \quad \text{ and } \quad \|y_2\|_{Y_T} \le C \|u\|_{[H^{\frac{3}{4}}(0, T)]^*}^2. 
\end{align}
Combining  \eqref{deltay-YT}, \eqref{deltay-YT-2}, and \eqref{deltay-YT-12} yields 
\begin{align}	 
\|\delta y\|_{Y_T} \leq 	 C \left( \|u\|_{[H^{\frac{3}{4}}(0, T)]^*} +  \|u\|_{[H^{\frac{3}{4}}(0, T)]^*}^2 \right) \left(  \|u\|_{[H^{\frac{3}{4}}(0, T)]^*}^2  + \|\delta y\|_{Y_T} \right) + C \|y_0\|_{L^2(0, 1)}.
	\end{align}
Consequently, there exists $\eps_1 \in (0, 1)$ depending only on $T$ such that if
\begin{align}	\label{lem-deltay-3}
\|u\|_{[H^{\frac{3}{4}}(0, T)]^*}  \le \eps_1
\end{align}
(condition \eqref{lem-deltay-3} will be assumed from now on in this proof), 
then
\begin{align}\label{lem-deltay-YT-p1}
	\|\delta y\|_{Y_T} \leq C \|u\|^3_{[H^{\frac{3}{4}}(0, T)]^*}  + C \|y_0\|_{L^2(0, 1)}, 
\end{align}
which is \eqref{lem-deltay-1}.

\medskip	
	
We next deal with \eqref{lem-deltay-2}. It follows from 	\eqref{yyx-y1y1x} that
\begin{align}	\label{yyx-y1y1x-3}
	 -y y_x + y_1 y_{1, x} = - h_x \quad \text{ in } (0, T) \times (0, 1),
\end{align}
where
\begin{align}	\label{yyx-y1y1x-3'}
h = (y_1 - U) y_2 + U y_2 + y_2^2/2 + (y_1 + y_2) \delta y + (\delta y)^2/2.
\end{align}
We derive from \eqref{yyx-y1y1x-3} and \Cref{lem-B2} that 
	\begin{align} \label{deltay-L2}
		\|\delta y\|_{L^2((0, T) \times (0, 1))} \leq   
		 C \left\| h_x \right\|_{\left(C \left([0, T], H^1_0(0, 1) \right)  \right)^*} +  C \|y_0\|_{L^2(0, 1)}. 
	\end{align}
	
We have
\be \label{lem-dysmall-p1}
 \left\| h_x \right\|_{\left(C \left([0, T], H^1_0(0, 1) \right)  \right)^*}  \leq C \| h  \|_{L^1((0, T), L^2(0, 1))}.
\ee
Since (see, e.g., \cite[Lemma 8.1]{Nguyen}) 
$$
\| (y_1 - U) y_2\|_{L^1((0, T), L^2(0, 1))} \le C  \|y_1 - U\|_{L^2((0, T) \times (0, 1))}^{\frac{1}{2}}  \|y_1 - U\|_{L^2((0, T), H^1(0, 1))}^{\frac{1}{2}}  \|y_2\|_{L^2((0, T) \times (0, 1))}, 
$$
it follows from \eqref{yyx-y1y1x-3'} and \eqref{lem-dysmall-p1} that  
\begin{multline}\label{deltay-L2-1}
 \left\| h_x \right\|_{\left(C \left([0, T], H^1_0(0, 1) \right)  \right)^*}  \leq C  \|y_1 - U\|_{L^2((0, T) \times (0, 1))}^{\frac{1}{2}}  \|y_1 - U\|_{L^2((0, T), H^1(0, 1))}^{\frac{1}{2}}  \|y_2\|_{L^2((0, T) \times (0, 1))} \\[6pt]
  + C \|U\|_{L^2(0, T)} \|y_2\|_{L^2((0, T) \times (0, 1))}  +  C \|y_2\|_{Y_T}^{\frac{1}{2}} \|y_2\|_{L^2((0, T) \times (0, 1))}^{\frac{3}{2}} \\[6pt]
   + C \left( \|y_1\|_{L^2((0, T), H^1(0, 1))} + \|y_2\|_{Y_T} + \|  \delta y\|_{Y_T} \right) \| \delta y \|_{L^2((0, T) \times (0, 1))}. 
\end{multline}

Combining \eqref{deltay-L2} and \eqref{deltay-L2-1} yields
\begin{multline}	\label{deltay-L2-2}
	\|\delta y\|_{L^2((0, T) \times (0, 1))} \leq 		C   \|y_1 - U\|_{L^2((0, T) \times (0, 1))}^{\frac{1}{2}} \|y_1 - U\|_{L^2((0, T), H^1(0, 1))}^{\frac{1}{2}}  \|y_2\|_{L^2((0, T) \times (0, 1))} \\[6pt]
  + C \|U\|_{L^2(0, T)} \|y_2\|_{L^2((0, T) \times (0, 1))}  +  C \|y_2\|_{Y_T}^{\frac{1}{2}} \|y_2\|_{L^2((0, T) \times (0, 1))}^{\frac{3}{2}} \\[6pt]
   + C \left( \|y_1\|_{L^2((0, T), H^1(0, 1))} + \|y_2\|_{Y_T} + \|  \delta y\|_{Y_T} \right) \| \delta y \|_{L^2((0, T) \times (0, 1))}    + C \|y_0\|_{L^2(0, 1)}.
\end{multline}

As a consequence of \eqref{deltay-YT-12}, \eqref{lem-deltay-YT-p1}, and \eqref{deltay-L2-2}, there exists $\varepsilon \in (0, \varepsilon_1)$ depending only on $T$ such that if 
\begin{align}	\label{choose-epsilonT}
		\|y_0\|_{L^2(0, 1)} + \| u\|_{[H^{\frac{3}{4}}(0, T)]^*} \le \eps,
\end{align}
then 
\begin{multline} \label{lem-deltay-4}
\|\delta y\|_{L^2((0, T) \times (0, 1))} \leq C  \|y_1 - U\|_{L^2((0, T) \times (0, 1))}^{\frac{1}{2}}  \|y_1 - U\|_{L^2((0, T), H^1(0, 1))}^{\frac{1}{2}}  \|y_2\|_{L^2((0, T) \times (0, 1))} \\[6pt]
  + C \|U\|_{L^2(0, T)} \|y_2\|_{L^2((0, T) \times (0, 1))}  +  C \|y_2\|_{Y_T}^{\frac{1}{2}} \|y_2\|_{L^2((0, T) \times (0, 1))}^{\frac{3}{2}} + C \|y_0\|_{L^2(0, 1)}.
\end{multline}

Applying \Cref{pro-y1} to $y_1$ and using \Cref{lem-y2small}, we have
\begin{align}	\label{lem-deltay-5}
\|y_1 - U\|_{L^2((0, T) \times (0, 1))} \leq C \|u \|_{[H^{\frac{5}{4}}(0, T)]^*},
\end{align}
\begin{multline}	\label{lem-deltay-6}
\|y_1 - U\|_{L^2((0, T), H^1(0, 1))} \\[6pt]
\leq C \left(\|y_1 - U\|_{L^2((0, T) \times (0, 1))} + \|y_{1,x}\|_{L^2((0, T) \times (0, 1))}  \right) 
\leq C \|u\|_{[H^{\frac{3}{4}}(0, T)]^*},
\end{multline}
\begin{align}	\label{lem-deltay-7}
 \|y_2\|_{L^2((0, T) \times (0, 1))} \leq C \|u\|_{[H^{\frac{5}{4}} (0, T)]^*} \|u\|_{[H^{\frac{3}{4}}(0, T)]^*}, \quad \mbox{ and } \quad \|y_2\|_{Y_T} \le C \|u\|_{[H^{\frac{3}{4}}(0, T)]^*}^2. 
\end{align}
Since
\begin{align}	\label{lem-deltay-8}
	\|U \|_{L^2(0, T)} \leq C \|u\|_{[H^1(0, T)]^*} \leq C \|u\|_{[H^{\frac{5}{4}}(0, T)]^*}^{\frac{1}{2}}   \|u\|_{[H^{\frac{3}{4}}(0, T)]^*}^{\frac{1}{2}}, 
\end{align}
we deduce from \eqref{lem-deltay-4}, \eqref{lem-deltay-5}, \eqref{lem-deltay-6}, \eqref{lem-deltay-7}, and \eqref{lem-deltay-8} that
\begin{align} \label{lem-deltay-L2-p2}
\|\delta y\|_{L^2((0, T) \times (0, 1))} \leq C \| u\|_{[H^{\frac{5}{4}}(0, T)]^*}^{\frac{3}{2}}  \| u\|_{[H^{\frac{3}{4}}(0, T)]^*}^{\frac{3}{2}} +  C \|y_0\|_{L^2(0, 1)}. 
\end{align}

The conclusion now follows from \eqref{lem-deltay-YT-p1} and \eqref{lem-deltay-L2-p2}.  The proof is complete. 
\end{proof}

As a direct consequence of  \Cref{pro-y1}, \Cref{lem-y2small}, and \Cref{lem-dysmall}, we have the following estimates for $y$.

\begin{corollary}\label{cor-y} Let $T > 0$, $u \in L^1(0, T)$, and $y_0 \in L^2(0, 1)$. Let $y \in Y_T$ be the unique solution of  \eqref{M-y}. 
There exist positive constants $\varepsilon$ and $C$ depending only on $T$ such that, if 
\begin{align}
\|y_0\|_{L^2(0, 1)} + \|u\|_{[H^{\frac{3}{4}}(0, T)]^*} \leq \varepsilon, 
\end{align}
then
\begin{align}	\label{cor-y-1}
	\|y\|_{L^2((0, T), H^1(0, 1))} \leq C \Big( \|u\|_{[H^{\frac{3}{4}}(0, T)]^*} + \|y_0\|_{L^2(0, 1)} \Big),
\end{align}
and
\begin{align}	\label{cor-y-2}
	\|  y\|_{L^2((0, T) \times (0, 1))} \leq  C \Big( \|u\|_{[H^{1}(0, T)]^*} +  \|y_0\|_{L^2(0, 1)} \Big).
\end{align}
\end{corollary}

We are in a position to prove \Cref{coercivity->obstruction}.
       
\subsection{Proof of Theorem \ref{coercivity->obstruction}} 	\label{sub-sec-mainproof}   

Let $\varepsilon_0 \in (0, 1)$ be arbitrarily small. Suppose by contradiction that there exist $\varepsilon \in (0, \varepsilon_0)$ and  $u \in L^1(0, T)$ with $\| u\|_{[H^{\frac{3}{4}}(0, T)]^*} < \varepsilon_0$ such that $y(T, \cdot) = 0$, where $y \in Y_T$ is the solution of \eqref{Burgersdetailed}. Extend $u(t)$ and $y (t, \cdot)$  by $0$ for $t > T$ and still denote the extensions by $u$ and $y$.  Then 
       \begin{align}   \label{Burgersdetailed2T}
            \begin{cases}
                y_t (t, x) - y_{xx} (t, x) + yy_x (t, x) = u(t) \quad &\text{ in } (0, +\infty) \times (0, 1), \\[6pt]
                y(t, 0) = y(t, 1) = 0 &\text{ in } (0, +\infty), \\[6pt]
                y(0, x) =  - \varepsilon \sin (k_0 \pi x) &\text{ in } (0, 1),
            \end{cases}
        \end{align}
and
\begin{align}	\label{y>T=0}
	y(t, \cdot) = 0 \mbox{ for } t \ge T. 
\end{align}
By integration by parts, we have that, for every $\tau \geq T$,
 \begin{align}	\label{step2-claim-neven}
           \int_0^{\tau} \int_0^1  y^2(t, x) \phi_{k_0, x}(t, x) dx dt = 	\varepsilon,      
        \end{align}
        where $\phi_{k_0}$ is defined in \eqref{def-phik}.

\medskip

    Let $y_1 \in C\left([0, +\infty), L^2(0, 1) \right) \cap L^2_{loc} \left((0, +\infty), H^1_0(0, 1) \right)$  be the solution of 
        \begin{align}	\label{step2-y1}
            \begin{cases}
                y_{1, t} (t, x) - y_{1, xx} (t, x) = u(t) \quad &\text{ in } (0, +\infty) \times (0, 1), \\[6pt]
                y_1(t,0) = y_1(t,1) = 0 &\text{ in } (0, +\infty), \\[6pt]
                y_1(0, x) = 0 &\text{ in } (0, 1).
            \end{cases}
        \end{align}             
      
It follows from \eqref{ipp-linear} that
        \begin{align}	\label{y12T-ortho}
        y_1 \left(2T, \cdot \right) \in \mathscr{M}^{\bot}, 
        \end{align}
where $\mathscr{M}$ is defined in \eqref{unreachable space}.  Using the moment method (see, e.g., \Cref{pro-heat-moment} in \Cref{sect-moment}), we deduce that there exists a real-valued control $u^* \in L^2(\mathbb{R})$ with $\mathrm{supp} \; u^* \subset [2T, 3T]$ such that 
        \begin{align}	\label{chooseu*}
            \|u^*\|_{L^2(2T, 3T)} \leq C \left\|y_1\left(2T, \cdot \right) \right\|_{L^2(0, 1)},
        \end{align}
        and 
        \begin{align}	\label{y1*(T)=0}
  			y_1^*(3T, \cdot) = 0,      
        \end{align}
 where $y_1^* \in C([2T, 3T], L^2(0, 1)) \cap L^2\left((2T,3T), H^1_0(0, 1) \right)$ is the solution of
        \begin{align}	\label{sys-y*}
        \begin{cases}
            y_{1, t}^*(t,x) - y_{1, xx}^*(t,x) = u^*(t) \quad &\text{ in } (2T, 3T) \times (0, 1), \\[6pt]
            y_{1}^*(t, 0) = y_1^*(t, 1) = 0 &\text{ in } (2T, 3T), \\[6pt]
            y_{1}^*(2T, x) = y_1(2T, x) &\text{ in } (0, 1). 
        \end{cases}
        \end{align}
Here and in what follows, $C$ denotes a positive constant depending only on $T$ and varies line by line.     
		
		\medskip
		
		Define 
		\begin{align}	\label{def-tildeu} 
			\tu = \begin{cases}
				u	 \quad 	&\text{ in } (0, 2T), \\[6pt]
				u^* \quad &\text{ in } (2T, 3T),
			\end{cases}		\quad \mbox{ and } \quad
			\ty = \begin{cases}
				y_1 \quad &\text{ in } [0, 2T] \times (0, 1), \\[6pt]
				y_1^* \quad &\text{ in } [2T, 3T] \times (0, 1).
			\end{cases}
		\end{align}		        
        Then $\tu \in L^1((0, 3T), \mR)$, and $\ty \in Y_{3T}$ is the solution of the system        
        \begin{align}	\label{tildey-1}
        \begin{cases}
            \ty_t (t, x) - \ty_{xx} (t, x)  = \tu(t) \quad &\text{ in } (0, 3T) \times (0, 1), \\[6pt]
            \ty(t, 0) = \ty(t, 1) = 0 &\text{ in } (0, 3T), \\[6pt]
            \ty(0, x) = 0 &\text{ in } (0, 1),
        \end{cases}
        \end{align}
     and, by \eqref{y1*(T)=0}, 
		\begin{align}        	
        	\ty(3T, \cdot) = 0.
		\end{align}				        	
        	 Applying \Cref{2obstruction} to $\ty$, we obtain
        \begin{align} 	\label{applyProp2.1}
            \int_{0}^{3T} \int_0^1 \ty^2(t, x) \phi_{k_0, x}(t, x) dx dt \leq - \alpha \|\tu \|^2_{[H^{\frac{5}{4}}(0, 3T)]^*},
        \end{align}
        where $\alpha$ is a fixed positive constant depending only on $T$.

\medskip

We claim that, for $\varepsilon_0$ sufficiently small, 
\begin{align}	\label{step4-claim}
	\varepsilon + \alpha \| \tu \|^2_{[H^{\frac{5}{4}}(0, 3T)]^*} \leq 	| I | + C\|u\|_{[H^1(0, T)]^*}^4 + C \varepsilon^2,  
\end{align}
where
\begin{align} \label{step4-I}
I  = \int_0^{2T} \int_0^1 \left(y^2(t, x) - y_1^2(t, x)\right) \phi_{k_0, x}(t, x) dx dt. 
\end{align}

Indeed, by \eqref{step2-claim-neven} and  \eqref{applyProp2.1}, 
\begin{align}	 \label{step4-p0}
\varepsilon + \alpha \|\tu \|_{[H^{\frac{5}{4}}(0, 3T)]^*}^2 \leq \int_0^{2T} \int_0^1 y^2(t, x) \phi_{k_0, x}(t, x) dx dt -   \int_{0}^{3T} \int_0^1 \ty^2(t, x) \phi_{k_0, x}(t, x)  dx dt.
\end{align}          
      Since
        \begin{multline*}
        		\int_0^{2T} \int_0^1 y^2(t, x) \phi_{k_0, x}(t, x) dx dt -   \int_{0}^{3T} \int_0^1 \ty^2(t, x) \phi_{k_0, x}(t, x)  dx dt \\
        		\mathop{=}^{\eqref{def-tildeu}} \int_0^{2T} \int_0^1 \left(y^2(t, x) - y_1^2(t, x)\right) \phi_{k_0, x}(t, x) dx dt  - \int_{2T}^{3T} \int_0^1 (y_1^*)^2(t, x) \phi_{k_0, x}(t, x) dx dt \\[6pt]
		= I - \int_{2T}^{3T} \int_0^1 (y_1^*)^2(t, x) \phi_{k_0, x}(t, x) dx dt, 
        \end{multline*}   
it follows from \eqref{step4-p0} that
\begin{align}
\label{step4-p1}
\varepsilon + \alpha \|\tu \|_{[H^{\frac{5}{4}}(0, 3T)]^*}^2 \leq   \left|I \right| + C \|y_1^* \|_{L^2((2T, 3T) \times (0, 1))}^2. 
\end{align}

We next estimate the last term of the right-hand side of \eqref{step4-p1}.  We have
\begin{align}	\label{step4-p3}
	\|y_1^*\|_{L^2((2T, 3T) \times (0, 1))}  \mathop{\leq}^{\eqref{sys-y*}} C \left( \|u^*\|_{L^2(2T, 3T)} + \|y_1(2T, \cdot) \|_{L^2(0, 1)} \right)  \mathop{\leq}^{\eqref{chooseu*}} C \|y_1(2T, \cdot) \|_{L^2(0, 1)}.
\end{align}
Since $y(t, \cdot)= 0$ for $t\ge T$, we have
\begin{align*} 
\| y_1(2T, \cdot)  \|_{L^2(0, 1)} = \| y_1(2T, \cdot)  - y(2T, \cdot) \|_{L^2(0, 1)}, 
\end{align*}
and, since $(y - y_1)_t - (y - y_1)_{xx} = -\partial_x(y^2)/2$, by \Cref{lem-fx-L1} in \Cref{wp-bg}, we obtain
\begin{align*} 
\| y_1(2T, \cdot)  - y(2T, \cdot) \|_{L^2(0, 1)} \le  C \Big( \eps + \| y^2\|_{L^1((0, T) \times (0, 1))}  \Big). 
\end{align*}
It follows that 
\begin{align}  \label{step4-p4}
\| y_1(2T, \cdot)  \|_{L^2(0, 1)} \le C \Big( \eps + \| y^2\|_{L^1((0, T) \times (0, 1))}  \Big). 
\end{align}
   
        Combining \eqref{step4-p3} and \eqref{step4-p4} yields 
 \begin{align}	\label{step4-p2}
	\|y_1^*\|^2_{L^2((2T, 3T) \times (0, 1))} \leq C \Big( \|y\|^4_{L^2((0, T) \times (0, 1))} + \eps^2 \Big).
\end{align}   
Using the fact that (see \Cref{cor-y})
         \begin{align*}
\|y\|_{L^2((0, T) \times (0, 1))} \leq C \|u\|_{[H^1(0, T)]^*} + C \varepsilon, 
       \end{align*}
we obtain  \eqref{step4-claim} from  \eqref{step4-p1} and
\eqref{step4-p2}.  
        
\medskip

Let $y_2 \in Y_{2T}$ be the solution of
\begin{align}
	    \begin{cases}
                y_{2,t} - y_{2, xx} = -y_1 y_{1,x}  \quad &\text{ in } (0, 2T) \times (0, 1), \\[6pt]
                y_2(t, 0) = y_2(t, 1) = 0 &\text{ in } (0, 2T), \\[6pt]
                y_2(0, x) = 0 &\text{ in } (0, 1),
            \end{cases}
\end{align}
and set
\begin{align}	\delta y = y - y_1 - y_2 \quad \text{ in } (0, 2T) \times (0, 1).
\end{align}    
Then $\delta y \in Y_{2T}$. Since $\mbox{supp } u \subset [0, T]$, we have
	\begin{align}	\label{proof-y1}
			\|y_1 \|_{L^2((0, 2T) \times (0, 1))} \mathop{\leq}^{\Cref{pro-y1}} C \|u\|_{[H^1(0, 2T)]^*} \leq C \|u\|_{[H^1(0, T)]^*}, 
		\end{align}
		and 
           \begin{align}	\label{proof-y2}
           		\|y_2 \|_{L^2((0, 2T) \times (0, 1))} \mathop{\leq}^{\Cref{lem-y2small}} C \|u\|_{[H^{\frac{5}{4}}(0, 2T)]^*} \|u\|_{[H^{\frac{3}{4}}(0, 2T)]^*} \le  C \|u\|_{[H^{\frac{5}{4}}(0, T)]^*} \|u\|_{[H^{\frac{3}{4}}(0, T)]^*}.
           \end{align}
For $\varepsilon_0$ sufficiently small, we have 
           \begin{multline}	\label{proof-deltay}
           		 \|\delta y \|_{L^2((0, 2T) \times (0, 1))} \mathop{\leq}^{\Cref{lem-dysmall}} C  \|u\|_{[H^{\frac{5}{4}}(0, 2T)]^*}^{\frac{3}{2}} \|u\|_{[H^{\frac{3}{4}}(0, 2T)]^*}^{\frac{3}{2}} + C \varepsilon \\[6pt]
		\leq C  \|u\|_{[H^{\frac{5}{4}}(0, T)]^*}^{\frac{3}{2}} \|u\|_{[H^{\frac{3}{4}}(0, T)]^*}^{\frac{3}{2}} + C \varepsilon.
           \end{multline}
     
     We next estimate $| I |$ where $I$ is defined in \eqref{step4-I}.    
Since
$$
y^2 - y_1^2 = \left(y - y_1\right)^2 + 2y_1 (y - y_1) = 2y_1 y_2  + \left(y_2 + \delta y \right)^2 +  2y_1 \delta y, 
$$
it follows that   
        \begin{align}
          I = I_1 + I_2 + I_3, \label{rewriteI}
        \end{align}             
        where
\begin{align*}
	I_1 = 2 \int_0^{2T} \int_0^1 y_1 y_2 \phi_{k_0, x} dx dt, \; \; I_2 = \int_0^{2T} \int_0^1 \left(y_2 + \delta y \right)^2 \phi_{k_0, x} dx dt,\; \;   I_3 = 2 \int_0^{2T} \int_0^1 y_1 \delta y \phi_{k_0, x} dx dt. 
\end{align*}  
      
We next estimate $I_1, I_2, I_3$. We first deal with $I_1$. By definition of $y_1$ and $y_2$, we have
        \begin{align}	
        y_1(t, \cdot) = y_1(t, 1 - \cdot)  \quad \mbox{ and } \quad y_2(t, \cdot) = - y_2 (t, 1- \cdot) \quad \mbox{ for } t \in (0, 2T).
        \end{align}
        Since $k_0$ is even, it follows that
        $$
        \int_0^1 y_1 (t, x) y_2(t, x) \cos (k_0 \pi x)  dx = 0 \quad \mbox{ for } t \in (0, 2T), 
        $$
which yields         
\begin{align}	\label{In2}
 I_1 = 0. 
        \end{align}        

Concerning $I_2$, we have
        \begin{align}   \label{In1}
           |I_2| \leq C \| y_2 + \delta y \|^2_{L^2((0, 2T)\times (0, 1))} 
            \mathop{\leq}^{ \eqref{proof-y2}, \eqref{proof-deltay}} C \|u \|_{[H^{\frac{5}{4}}(0, T)]^*}^2 \|u \|_{[H^{\frac{3}{4}}(0, T)]^*}^2 + C \varepsilon^2.
        \end{align}    

Finally, we estimate $I_3$. We have
\begin{multline}
	|I_3| \leq C \|y_1\|_{L^2((0, 2T) \times (0, 1))} \|\delta y\|_{L^2((0, 2T) \times (0, 1))} \\
	\overset{\eqref{proof-y1}, \eqref{proof-deltay}}{\leq} C  \|u\|_{[H^1(0, T)]^*}   \|u\|_{[H^{\frac{5}{4}}(0, T)]^*}^{\frac{3}{2}} \|u\|_{[H^{\frac{3}{4}}(0, T)]^*}^{\frac{3}{2}} + C \varepsilon \|u\|_{[H^1(0, T)]^*}. \label{In3-1}
\end{multline}
Since
\begin{align}	\label{In2interpolation}
	\|u\|_{[H^1(0, T)]^*} \leq C \|u\|_{[H^{\frac{5}{4}}(0, T)]^*}^{\frac{1}{2}} \|u\|_{[H^{\frac{3}{4}}(0, T)]^*}^{\frac{1}{2}}, 
\end{align}
we deduce from \eqref{In3-1} and \eqref{In2interpolation} that 
	\begin{align}	\label{In3-2}
	|I_3| \leq  C \|u\|_{[H^{\frac{5}{4}}(0, T)]^*}^2 \|u\|_{[H^{\frac{3}{4}}(0, T)]^*}^2 + C \varepsilon \|u\|_{[H^1(0, T)]^*}.
	\end{align}

Combining \eqref{rewriteI}, \eqref{In2}, \eqref{In1}, and \eqref{In3-2}, we obtain \begin{align} \label{Iny-y1}
   |I| \leq C \|u\|_{[H^{\frac{5}{4}}(0, T)]^*}^2 \|u\|_{[H^{\frac{3}{4}}(0, T)]^*}^2 + C \varepsilon \|u\|_{[H^1(0, T)]^*} + C \eps^2. 
    \end{align}

  From \eqref{step4-claim} and \eqref{Iny-y1}, we obtain, for $\varepsilon_0$ sufficiently small,
    \begin{multline} \label{step-final-p1}
      \varepsilon + \alpha \| \tu \|^2_{[H^{\frac{5}{4}}(0, 3T)]^*} \\[6pt]
      \le  C \|u\|_{[H^{\frac{5}{4}}(0, T)]^*}^2 \|u\|_{[H^{\frac{3}{4}}(0, T)]^*}^2 + C \varepsilon \|u\|_{[H^1(0, T)]^*} + C \eps^2 + C  \|u\|_{[H^1(0, T)]^*}^4. 
    \end{multline}
    Since $\| u\|_{[H^{\frac{3}{4}}(0, T)]^*} < \eps_0$, it follows from \eqref{In2interpolation} and \eqref{step-final-p1} that
\be
\eps +  \| \tu \|_{[H^{\frac{5}{4}}(0, 3T)]^*}^2 \le C   \|u\|_{[H^{\frac{5}{4}}(0, T)]^*}^2 \|u\|_{[H^{\frac{3}{4}}(0, T)]^*}^2.
\ee
Since $\|u\|_{[H^{\frac{5}{4}}(0,T)]^*} \leq C \|\tu \|_{[H^{\frac{5}{4}}(0, 3T)]^*}$ by \eqref{def-tildeu}, it follows that 
\be
\eps + \| \tu \|_{[H^{\frac{5}{4}}(0, 3T)]^*}^2 \le C \| \tu \|_{[H^{\frac{5}{4}}(0, 3T)]^*}^2 \| u \|_{[H^{\frac{3}{4}}(0, T)]^*}^2 \le C \eps_0^2 \| \tu \|_{[H^{\frac{5}{4}}(0, 3T)]^*}^2, 
\ee
which yields, since $\eps_0$ is sufficiently small, 
\begin{align}
\eps = 0.
\end{align}
We have a contradiction. 

\medskip
The proof of  \Cref{coercivity->obstruction} is complete. \qed

\appendix

\section{On the heat equation and the Burgers equation}
 \label{wp-bg} 
 \label{controlheatu(t)}

In this section, we recall basic facts on the linear heat equation and the Burgers equation. We first deal with the heat equation, for $T>0$, $f \in L^1((0, T), L^2(0, 1))$, and $y_0 \in L^2(0, 1)$: 
\begin{align}	
	\begin{cases}	\label{ap-heat-sys}
		y_t - y_{xx} = f \quad &\text{ in } (0, T) \times (0, 1), \\[6pt]
		y(t, 0) = y(t, 1) = 0 &\text{ in } (0, T), \\[6pt]
		y(0, x) = y_0(x) &\text{ in } (0, 1).
	\end{cases}
\end{align}

We recall the space $Y_T$ defined in \eqref{YT}. In what follows, we also consider the space
\begin{align*}
	X_T = C([0, T], H^1_0(0, 1)) \cap L^2((0, T), H^2(0, 1)),
\end{align*}
equipped with the norm
\begin{align*}
	 \|y\|_{X_T} = \|y\|_{ C([0, T], H^1_0(0, 1))} + \|y\|_{ L^2((0, T), H^2(0, 1))}.
\end{align*}

The following standard definition of weak solutions of \eqref{ap-heat-sys} is used. 

\begin{definition}	\label{def-heat-weaksolution}	Let $T > 0$, $f \in L^1((0, T), L^2(0, 1))$, and $y_0 \in L^2(0, 1)$. A function $y \in Y_T$ is a solution of \eqref{ap-heat-sys} if for every $\tau \in [0, T]$ and $\phi \in X_T \cap H^1((0, T) \times (0, 1))$,
\begin{multline} 	\label{heat-weaksolution}
	\int_0^1 y(\tau, x) \phi(\tau, x) dx - \int_0^1 y_0(x) \phi(0, x) dx - \int_0^{\tau} \int_0^1 y(t, x) \; (\phi_t + \phi_{xx} )(t, x) \; dx dt \\
	 = \int_0^{\tau} \int_0^1 f(t, x) \phi(t, x) dx dt.
\end{multline}
\end{definition}

\begin{remark}	\label{rem-strongsol=>weaksol} \rm
Using the integration by parts formula, one can check that if $y$ is smooth enough and satisfies \eqref{ap-heat-sys} in the classical sense, then $y$ is a solution of \eqref{ap-heat-sys} in the sense of  \Cref{def-heat-weaksolution}. 
\end{remark}

\medskip 
The following result on \eqref{ap-heat-sys} is standard.  

\begin{lemma}\label{ap-heat-fx}
Let $T > 0$, $f  \in L^1((0, T), L^2(0, 1))$, and $y_0 \in L^2(0, 1)$. Then  system \eqref{ap-heat-sys} has a unique solution $y \in Y_T$. Moreover, there exists a positive constant $C$ depending only on $T$ such that
\begin{align}	\label{ap-heat-fx-L2Linfty}
\|y\|_{Y_T} \leq C \left(\|f\|_{L^1((0, T), L^2(0, 1))} + \|y_0\|_{L^2(0, 1)} \right).
\end{align}
In addition, if $f \in L^2((0, T) \times (0, 1))$ and $y_0 \in H^1_0(0, 1)$, then $y \in X_T$. Moreover, there exists a positive constant $C$ depending only on $T$ such that
\begin{align}	\label{ap-heat-fx-XT}
	\|y\|_{X_T} \leq C \left(\|f\|_{L^2((0, T) \times (0, 1))} + \|y_0\|_{H^1_0(0, 1)} \right).
\end{align}
\end{lemma}

\begin{proof} Assertion \eqref{ap-heat-fx-XT} is known, see, e.g., \cite[Section 7.1]{Evans}.  Assertion \eqref{ap-heat-fx-L2Linfty} is quite standard but we cannot find a reference for it;  many standard references consider $\|f\|_{L^2((0, T), L^2(0, 1))}$ instead of  $\|f\|_{L^1((0, T), L^2(0, 1))}$, see, e.g., \cite[Section 7.1]{Evans}.  The proof is almost the same. For the convenience of the readers, we show \eqref{ap-heat-fx-L2Linfty} when $f$, $y_0$, and $y$ are smooth enough. For every $\tau \in [0, T]$, we multiply \eqref{ap-heat-sys} by $y$ and integrate by parts on $(0, \tau) \times (0, 1)$ to obtain 
\begin{align} \label{lem-ap1-p1}
\frac{1}{2} \left(\int_0^1 y^2(\tau, x) dx - \int_0^1 y_0^2(x) dx \right) + \int_0^{\tau} \int_0^1 y^2_x(t, x) dx dt = \int_0^{\tau} \int_0^1 f(t, x) y(t, x) dx dt.
\end{align}
This implies 
	\begin{align*}
		\|y\|_{C([0, T], L^2(0, 1))}^2 &\leq \|y_0\|^2_{L^2(0, 1)} + 2 \|f\|_{L^1((0, T), L^2(0, 1))}  \|y\|_{C([0, T], L^2(0, 1))} \\
		&\leq \|y_0\|_{L^2(0, 1)}^2 + \frac{1}{2} \|y\|^2_{C([0, T], L^2(0, 1))} + 2 \|f\|_{L^1((0, T), L^2(0, 1))}^2 , 
	\end{align*}
	so
	\begin{align}	\label{lem-yC0L2}
		\|y\|_{C([0, T], L^2(0, 1))} \leq 2 \left(\|y_0\|_{L^2(0, 1)} + \|f\|_{L^1((0, T), L^2(0, 1))}  \right).
	\end{align}
Combining \eqref{lem-ap1-p1} and \eqref{lem-yC0L2}, we obtain 
\begin{align*}
\| y_x \|_{L^2((0, T) \times (0, 1)}^2 &\leq  \frac{1}{2} \|y_0\|^2_{L^2(0, 1)} + 2\|f\|_{L^1((0, T), L^2(0, 1))}   \left(\|y_0\|_{L^2(0, 1)} +\|f\|_{L^1((0, T), L^2(0, 1))}  \right)  \\
		&\leq 2 \left(\|y_0\|_{L^2(0, 1)} + \|f\|_{L^1((0, T), L^2(0, 1))}  \right)^2, 
\end{align*}
so
\begin{align}
	\|y\|_{L^2((0, T), H^1_0(0, 1))} \leq  2 \left(\|y_0\|_{L^2(0, 1)} + \|f\|_{L^1(0, T), L^2(0, 1)} \right).  \label{lem-L2H10} 
\end{align}
Estimate \eqref{ap-heat-fx-L2Linfty} now follows from \eqref{lem-yC0L2} and \eqref{lem-L2H10}.  
\end{proof}

The following lemma, which estimates the $L^2$ norm of the solution $y$ of \eqref{ap-heat-sys}, is a consequence of \eqref{ap-heat-fx-XT}. 

\begin{lemma} \label{lem-B2}
Let $T > 0$, $f \in L^1((0, T), L^2(0, 1))$, and $y_0 \in L^2(0, 1)$. Let $y \in Y_T$ be the unique solution of \eqref{ap-heat-sys}. Then there exists a positive constant $C$ depending only on $T$ such that
\begin{align}	\label{ap-heat-fx-L2}
	\|y\|_{L^2((0, T) \times (0, 1))} \leq C  \left(\|f\|_{\left(C \left([0, T], H^1_0(0, 1) \right)  \right)^*} + \|y_0\|_{H^{-1}(0, 1)} \right). 
\end{align}

\end{lemma}

One has the following result when the source term $f$ is supported in $[0, T/2] \times [0, 1]$.

 \begin{lemma} \label{lem-fx-L1} Let $T > 0$,  $f \in L^1((0, T),  H^1(0, 1))$, and $y_0 \in L^2(0, 1)$. Suppose that $\mbox{supp } f \subset [0, T/2] \times [0, 1]$. Let $y \in Y_T$ be the unique solution of the system 
 \begin{align}
	    \begin{cases}
                y_{t}(t, x) - y_{xx} (t, x)= f_{x}(t, x)  \quad &\text{ in } (0, T) \times (0, 1), \\
                y(t, 0) = y(t, 1) = 0 &\text{ in } (0, T), \\
                y(0, x) = y_0(x)  &\text{ in } (0, 1).
            \end{cases}
\end{align}
Then there exists a positive constant $C$ depending only on $T$ such that
$$
\| y(T, \cdot)\|_{L^2(0, 1)} \le C \Big( \| f \|_{L^1((0, T) \times (0, 1))} + \| y_0 \|_{L^2(0, 1)} \Big). 
$$
\end{lemma}

\begin{proof} By \Cref{ap-heat-fx}, without loss of generality, one can assume that $y_0 = 0$. Then 
\begin{align*}
y(T, x) = - \sum_{k = 1}^{+\infty} \left( \sqrt{2} k \pi e^{-  \pi^2 k^2 T} \int_0^T \int_0^1 f(t, y) e^{\pi^2 k^2 t} \cos (k \pi y) \, dy \, dt \right) \sqrt{2} \sin (k \pi x) \mbox{ in } (0, 1). 
\end{align*}
This implies 
\begin{align*}
\| y(T, \cdot) \|_{L^2(0, 1)}^2 \le C \sum_{k = 1}^{+\infty} k^2 e^{- 2 \pi^2 k^2 T} \left| \int_0^T \int_0^1 f(t, y) e^{\pi^2 k^2 t} \cos (k \pi y) \, dy \, dt \right|^2. 
\end{align*}
Since $\mathrm{supp} \; f \subset [0, T/2] \times [0, 1]$, we derive that 
\begin{align*}
\| y(T, \cdot) \|_{L^2(0, 1)}^2 \leq C \left( \sum_{k = 1}^{+\infty} k^2 e^{- \pi^2 k^2 T} \right) \| f\|_{L^1((0, T) \times (0, 1))}^2 \leq C \|f\|^2_{L^1((0, T) \times (0, 1))}. 
\end{align*}
The conclusion follows. 
\end{proof}

 We next deal with the Burgers equation, for $T > 0$, $f \in L^1((0, T), L^2(0, 1))$, and $y_0 \in L^2(0, 1)$: 
\begin{align}		\label{ap-BG-sys}
	\begin{cases}
		y_t - y_{xx} + yy_x = f \quad &\text{ in } (0, T) \times (0, 1), \\[6pt]
		y(t, 0) = y(t, 1) = 0 &\text{ in } (0, T), \\[6pt]
		y(0, x) = y_0(x) &\text{ in } (0, 1).
	\end{cases}
\end{align}
Similar to \Cref{def-heat-weaksolution}, in this paper, we use the following standard definition of weak solutions of \eqref{ap-BG-sys}.

\begin{definition} \label{ap-def-sol-BG-0}	Let $T > 0$, $f \in L^1((0, T), L^2(0, 1))$, and $y_0 \in L^2(0, 1)$. A function $y \in Y_T$ is a solution of \eqref{ap-BG-sys} if for every $\tau \in [0, T]$ and $\phi \in X_T \cap H^1((0, T) \times (0, 1))$,
\begin{multline}	\label{ap-def-sol-BG}
	\int_0^1 y(\tau, x) \phi(\tau, x) dx - \int_0^1 y_0(x) \phi(0, x) dx - \int_0^{\tau} \int_0^1 y(t, x) \; (\phi_t + \phi_{xx})(t, x) \; dx dt \\
	 = \int_0^{\tau} \int_0^1 f(t, x) \phi(t, x) dx dt -  \int_0^{\tau} \int_0^1 y(t, x) y_x(t, x) \phi(t, x) dx dt.
\end{multline}
\end{definition}

The following result on \eqref{ap-BG-sys} is standard.  

\begin{lemma}\label{lem-ap1} Let $ T > 0$, $f \in L^1((0, T), L^2(0, 1))$, and  $y_0 \in L^2(0, 1)$. Then system \eqref{ap-BG-sys} has a unique solution $y \in Y_T$. Moreover, there exists a positive constant $C$ depending only on $T$ such that
\begin{align}\label{lem-ap1-cl}
\| y \|_{Y_T} \leq C \Big(  \|  f \|_{L^1((0, T), L^2(0, 1))}  + \| y_0 \|_{L^2(0, 1)} \Big). 
\end{align}
\end{lemma}

\section{Proof of Lemma \ref{lem-Ueandu}} \label{sect-lem-Ueandu}
We prove \eqref{lem-Ueandu-conclusion}. We have
\begin{align}
 \|U\|_{[H^{s - 1}(0, T)]^*} =  \sup_{\varphi \in H^{s - 1}(0, T), \| \varphi \|_{H^{s-1}(0, T)} \le 1} \left| \int_0^T U(t) \varphi(t) dt \right|. 	\label{lem-Ueandu-1}
\end{align}

Fix $\varphi \in H^{s - 1}(0, T)$.  Set 
$$
\Phi (t) = \int_t^T \varphi(s) \, ds \quad \text{ for } t \in [0, T]. 
$$
We have 
$$
\int_0^T U(t) \varphi(t) dt  = - \int_0^T U(t) \Phi'(t) dt = \int_0^T u(t) \Phi(t) dt. 
$$
It follows that 
$$
\Big| \int_0^T U(t) \varphi(t) dt \Big| \le \| u \|_{[H^{s}(0, T)]^*} \| \Phi \|_{H^{s}(0, T)}. 
$$
Since $\Phi' = - \varphi$ in $[0, T]$ and $\|\Phi\|_{L^2(0, T)} \le C \| \varphi \|_{L^2(0, T)}$, we have
$$
\| \Phi \|_{H^{s}(0, T)} \le C \| \varphi\|_{H^{s-1}(0, T)}. 
$$
Therefore,
$$
	\left| \int_0^T U(t) \varphi(t) dt \right| \leq C \|u \|_{[H^s(0, T)]^*} \| \varphi \|_{H^{s - 1}(0, T)}.
$$
Assertion \eqref{lem-Ueandu-conclusion} now follows from \eqref{lem-Ueandu-1}. \qed

\section{A controllability result for the heat equation via the moment method} \label{sect-moment}

Consider the control system 
\begin{align} 	\label{LinearBurgers}
\begin{cases}  
	y_t(t,x) - y_{xx}(t,x) = u(t) 	\quad &\text{ in } (0, T) \times (0, 1), \\[6pt]
	y(t, 0) = y(t, 1) = 0 &\text{ in } (0, T), \\[6pt]
	y(0, x) = y_0(x) &\text{ in } (0, 1).
\end{cases}
\end{align}
In this section, we study the initial data $y_0$ which can be steered to $0$ at time $T$ by using controls $u$ in $L^2(0, T)$.  Here is the main result.  

\begin{proposition}	\label{pro-heat-moment} 
Let $T > 0$.  Set 
\begin{align*}
    \mathscr{M} = \overline{\mathrm{span} \{\sin (k \pi x) \; | \;k \in \mathbb{N}^*, \; k \text{ even} \}}^{\| \|_{L^2(0, 1)}} \subset L^2((0, 1), \mR).
\end{align*}
There exists a bounded linear operator 
	\begin{align*}	
	\Gamma: \mathscr{M}^{\bot} \longrightarrow L^2((0, T), \mR)
	\end{align*}
	 such that for every $y_0 \in \mathscr{M}^{\bot}$, we have
	 \begin{align*}
	 	y(T, \cdot) = 0,
	\end{align*}	  
	 where $y \in Y_T$ is the solution of \eqref{LinearBurgers} with $u = \Gamma(y_0)$.
\end{proposition}

\begin{proof}[Proof of \Cref{pro-heat-moment}]  Let $y_0 \in L^2(0,1)$ and  $u \in L^2(0, T)$. Let $y \in Y_T$ be the solution of \eqref{LinearBurgers}. We have, for $k \in \mN^*$,
 \begin{multline*}  
    e^{ \pi^2 k^2 T} \int_0^1  y(T, x) \sin (k \pi x)  dx - \int_0^1 y_0(x) \sin (k \pi x)  dx =  \left(\int_0^1 \sin (k \pi x)  dx \right) \left(\int_0^T u(t) e^{\pi^2 k^2 t} dt \right) .
\end{multline*}
Note that for every $k \in \mN^*$,
$$
\int_0^1 \sin (k \pi x)  dx =  \frac{1 - (-1)^k}{k \pi}.
$$
One can thus steer $y_0$ to 0 at time $T$ if and only if $y_0 \in \mathscr{M}^{\bot}$ and for every odd $k \in \mN^*$,
\begin{align*}
 \int_0^T u(t) e^{\pi^2 k^2 t} dt = c_k , 
\end{align*}
where 
$$
c_k = -  \frac{k \pi}{2}  \int_0^1 y_0(x) \sin (k \pi x)  dx.  
$$
The proof is now based on the standard moment method. For example, one can follow almost the same arguments as in \cite{Tenenbaum-Tucsnak}. The details are omitted. 
\end{proof}

\providecommand{\bysame}{\leavevmode\hbox to3em{\hrulefill}\thinspace}
\providecommand{\MR}{\relax\ifhmode\unskip\space\fi MR }
\providecommand{\MRhref}[2]{%
  \href{http://www.ams.org/mathscinet-getitem?mr=#1}{#2}
}
\providecommand{\href}[2]{#2}

\end{document}